\documentclass[12pt]{article}

\usepackage{amssymb, amsmath, amsthm, amscd}

\usepackage[pdftex]{graphics}

\usepackage[utf8]{inputenc}

\usepackage[all,cmtip]{xy}

\usepackage{enumitem}

\usepackage{setspace}

\usepackage{mathdots}

\usepackage[mathscr]{euscript}

\usepackage[colorlinks=true]{hyperref}

\hypersetup{colorlinks   = true}

\hypersetup{linkcolor=blue}

\usepackage{tikz}

\usetikzlibrary{cd}

\DeclareMathOperator{\mult}{mult}

\begin{document}

\newtheorem{The}{Theorem}[section]

\newtheorem{Lem}[The]{Lemma}

\newtheorem{Prop}[The]{Proposition}

\newtheorem{Cor}[The]{Corollary}

\newtheorem{Rem}[The]{Remark}

\newtheorem{Obs}[The]{Observation}

\newtheorem{SConj}[The]{Standard Conjecture}

\newtheorem{Titre}[The]{\!\!\!\!}

\newtheorem{Conj}[The]{Conjecture}

\newtheorem{Question}[The]{Question}

\newtheorem{Prob}[The]{Problem}

\newtheorem{Def}[The]{Definition}

\newtheorem{Not}[The]{Notation}

\newtheorem{Claim}[The]{Claim}

\newtheorem{Conc}[The]{Conclusion}

\newtheorem{Ex}[The]{Example}

\newtheorem{Fact}[The]{Fact}

\newtheorem{Formula}[The]{Formula}

\newtheorem{Formulae}[The]{Formulae}

\newtheorem{The-Def}[The]{Theorem and Definition}

\newtheorem{Prop-Def}[The]{Proposition and Definition}

\newtheorem{Cor-Def}[The]{Corollary and Definition}

\newtheorem{Conc-Def}[The]{Conclusion and Definition}
\theoremstyle{remark}
\newtheorem{Terminology}[The]{Note on terminology}

\newcommand{\C}{\mathbb{C}}

\newcommand{\R}{\mathbb{R}}

\newcommand{\N}{\mathbb{N}}

\newcommand{\Z}{\mathbb{Z}}

\newcommand{\Q}{\mathbb{Q}}

\newcommand{\Proj}{\mathbb{P}}

\newcommand{\Rc}{\mathcal{R}}

\newcommand{\Oc}{\mathcal{O}}

\newcommand{\Vc}{\mathcal{V}}

\newcommand{\Id}{\operatorname{Id}}

\newcommand{\pr}{\operatorname{pr}}

\newcommand{\rk}{\operatorname{rk}}

\newcommand{\im}{\operatorname{im}}

\newcommand{\del}{\partial}

\newcommand{\delbar}{\bar{\partial}}

\newcommand{\Cdot}{{\raisebox{-0.7ex}[0pt][0pt]{\scalebox{2.0}{$\cdot$}}}}

\newcommand\nilm{\Gamma\backslash G}

\newcommand\frg{{\mathfrak g}}

\newcommand{\fg}{\mathfrak g}

\newcommand{\Oh}{\mathcal{O}}

\newcommand{\Kur}{\operatorname{Kur}}

\newcommand\gc{\frg_\mathbb{C}}

\newcommand\jonas[1]{{\textcolor{green}{#1}}}

\newcommand\luis[1]{{\textcolor{red}{#1}}}

\newcommand\dan[1]{{\textcolor{blue}{#1}}}

\begin{center}

{\Large\bf Higher-Page Hodge Theory of Compact Complex Manifolds}

\end{center}

\begin{center}

{\large Dan Popovici, Jonas Stelzig and Luis Ugarte}

\end{center}

\vspace{1ex}

\noindent{\small{\bf Abstract.} On a compact $\partial\bar\partial$-manifold $X$, one has the Hodge decomposition: the de Rham cohomology groups split into subspaces of pure-type classes as $H_{dR}^k (X)=\oplus_{p+q=k}H^{p,q}(X)$, where the $H^{p,q}(X)$ are canonically isomorphic to the Dolbeault cohomology groups $H_{\delbar}^{p,q}(X)$. For an arbitrary nonnegative integer $r$, we introduce the class of page-$r$-$\partial\bar\partial$-manifolds by requiring the analogue of the Hodge decomposition to hold on a compact complex manifold $X$ when the usual Dolbeault cohomology groups $H^{p,\,q}_{\bar\partial}(X)$ are replaced by the spaces $E_{r+1}^{p,\,q}(X)$ featuring on the $(r+1)$-st page of the Fr\"olicher spectral sequence of $X$. The class of page-$r$-$\partial\bar\partial$-manifolds coincides with the usual class of $\partial\bar\partial$-manifolds when $r=0$ but may increase as $r$ increases. We give two kinds of applications. On the one hand, we give a purely numerical characterisation of the page-$r$-$\partial\bar\partial$-property in terms of dimensions of various cohomology vector spaces. On the other hand, we obtain several classes of examples, including all complex parallelisable nilmanifolds and certain families of solvmanifolds and abelian nilmanifolds. Further, there are general results about the behaviour of this new class under standard constructions like blow-ups and deformations.

\vspace{2ex}


\vspace{2ex}

\section{Introduction}\label{section:Introduction} Let $X$ be an $n$-dimensional compact complex manifold. Recall the following notion that goes back to Deligne-Griffiths-Morgan-Sullivan [15] in the form (equivalent to that in [15]) and with the name given in [30, Definition 1.6].

\begin{Def}\label{Def:dd-bar-lemma} A compact complex manifold $X$ is said to be a $\partial\bar\partial$-{\bf manifold} if for any $d$-closed {\it pure-type} form $u$ on $X$, the following exactness properties are equivalent: \\

\hspace{10ex} $u$ is $d$-exact $\Longleftrightarrow$ $u$ is $\partial$-exact $\Longleftrightarrow$ $u$ is $\bar\partial$-exact $\Longleftrightarrow$ $u$ is $\partial\bar\partial$-exact.

\end{Def}

The classical {\bf $\partial\bar\partial$-lemma} asserts that every compact K\"ahler manifold is a $\partial\bar\partial$-manifold. More generally, thanks to [15], every {\it class ${\cal C}$ manifold} (i.e. every compact complex manifold bimeromorphically equivalent to a compact K\"ahler manifold) is a $\partial\bar\partial$-manifold. However, there exist many $\partial\bar\partial$-manifolds that are not of {\it class ${\cal C}$} (see e.g. [11], [24], [42], [30, Obs. 4.10], [3], [5, Thm 3.8], [16]).

On the other hand, every $\partial\bar\partial$-manifold admits canonically, namely in a way that depends only on the complex structure and does not involve arbitrary choices of metrics or other objects, a Hodge decomposition of the de Rham cohomology into a direct sum of subspaces of pure-type classes $H_{dR}^k(X)=\bigoplus _{p+q=k}H^{p,q}(X)$, with canonical isomorphisms $H_{\delbar}^{p,q}(X)\cong H^{p,q}(X)$, and the conjugation induces the Hodge symmetry, i.e. an antilinear isomorphism $H^{p,q}_{\delbar}(X)\cong H_{\delbar}^{q,p}(X)$ (accounting for the fact that some authors call these manifolds {\it cohomologically K\"ahler}).  In particular, the Fr\"olicher spectral sequence (FSS) of any $\partial\bar\partial$-manifold degenerates at the first page. However, the converse fails. (Indeed, as is well known, any non-K\"ahler compact complex surface provides a counter-example to the converse.) Meanwhile, $\partial\bar\partial$-manifolds also have good deformation and modification properties. For a review of these and some further properties of $\partial\bar\partial$-manifolds, see e.g. [30]. Let us now mention only a few of these for the reader's convenience:

\begin{enumerate}
	\item[(1)] The $\partial\bar\partial$-property is {\it deformation open} in the following sense: if $(X_t)_{t\in B}$ is a holomorphic family of compact complex manifolds $X_t$ parametrised by an open disc $B\subset\C$ about the origin (or by any complex manifold $B$), whenever $X_0$ (or any fibre $X_{t_0}$ with $t_0\in B$) is a $\partial\bar\partial$-manifold, every $X_t$ with $t\in B$ sufficiently close to $0$ (resp. to $t_0\in B$) is again a $\partial\bar\partial$-manifold. (See [42] or [8].)

\item[(2)] The $\partial\bar\partial$-property is {\it stable under contractions} in the following sense: if $f:\widetilde{X}\longrightarrow X$ is a holomorphic bimeromorphic map (i.e. a modification) between compact complex manifolds and if $\widetilde{X}$ is a $\partial\bar\partial$-manifold, then $X$ is again a $\partial\bar\partial$-manifold. (See [15, Theorem 5.22].) In particular, $f$ may be the blow-up of $X$ along a smooth centre $Z$. However, it is still unknown whether $X$ being a $\partial\bar\partial$-manifold implies that $\widetilde{X}$ is a $\partial\bar\partial$-manifold in full generality. It has recently been shown (see e.g. [6], [37], [38], [35], [43] or also \S \ref{section:geom-prop_page-r-ddbar}) that, on the one hand, $\widetilde{X}$ is a $\partial\bar\partial$-manifold whenever both $X$ and $Z$ are, and, on the other hand, that the $\partial\bar\partial$-property is {\it stable under blow-ups with smooth centres} if and only if it is inherited by any submanifold of a $\partial\bar\partial$-manifold. So, the remaining open problem in this direction is this {\it submanifold heredity} issue.

\item[(3)] The $\partial\bar\partial$-property of compact Calabi-Yau manifolds implies the {\it unobstructedness} of the small deformations of the complex structure in the following sense: if $X$ is a compact $\partial\bar\partial$-manifold whose canonical bundle $K_X$ is trivial, the Kuranishi family of $X$ is unobstructed. (This statement in the case where $X$ is K\"ahler is called the Bogomolov-Tian-Todorov Theorem -- see [9], [39], [40].)

\end{enumerate}

\vspace{2ex}

One of our main goals in this work is to relax the notion of $\partial\bar\partial$-manifold to accommodate a Hodge theory involving the higher pages of the Fr\"olicher spectral sequence (FSS) of $X$ when degeneration does not occur at the first page. As a result, we enhance the class of $\partial\bar\partial$-manifolds to classes of manifolds that seem worthy of further attention. This is motivated by the existence of many well-known non-K\"ahler and even non-$\partial\bar\partial$ compact complex manifolds, such as the Iwasawa manifold and higher-dimensional analogues thereof, whose Fr\"olicher spectral sequence degenerates only at the second page or later. Thus, we aim for a unified Hodge theoretical treatment of as large a class of compact complex manifolds as possible.

This approach enables us to get exact analogues of the usual Hodge decomposition and symmetry properties for what we call {\bf page-$(r-1)$-$\partial\bar\partial$-manifolds} using the spaces $E_r^{p,\,q}(X)$ featuring on the $r$-th page of the FSS, for some given $r\geq 2$, rather than the usual spaces $E_1^{p,\,q}(X)=H^{p,\,q}_{\bar\partial}(X)$ of the Dolbeault cohomology. The standard notion of $\partial\bar\partial$-manifold coincides with that of {\bf page-$0$-$\partial\bar\partial$-manifold}, while the class of {\bf page-$r$-$\partial\bar\partial$-manifolds}, our main find in this work, may grow as $r\in\N$ increases.

\subsection{Overview of the results}\label{subsection:main_results-statements_intro}

The following statement sums up several results and definitions of sections \ref{section:page-r-dd-bar} and \ref{section:page-r}. 
\begin{The-Def}\label{The-Def:main-def-prop_introd} Let $X$ be a compact complex manifold with $\mbox{dim}_\C X=n$. Fix an arbitrary $r\in\N^\star$. The following statements are equivalent.
\begin{enumerate}

\item[(1)] $X$ has the {\bf $E_r$-Hodge Decomposition} property, i.e. for every $k$, there is a $d$-closed representative of any class in $E_r^{p,q}(X)$ and sending such a representative to its de Rham class induces a \textbf{well-defined} isomorphism 
$$\bigoplus_{p+q=k}E_r^{p,\,q}(X)\longrightarrow H^k_{dR}(X,\,\C).$$

\item[(2)] The Fr\"olicher spectral sequence of $X$ {\bf degenerates at $E_r$} and the Hodge filtration induces a pure Hodge structure on the de Rham cohomology in every degree.

\item[(3)] The double complex $(\bigoplus_{p,q\in\Z}\mathcal{C}^\infty_{p,\, q}(X),\partial,\bar\partial)$ of smooth, complex-valued differential forms is isomorphic to a direct sum of indecomposable double complexes of the following types: {\bf squares}, namely double complexes of the shape 
\[
\begin{tikzcd}
	\langle\delbar a\rangle\ar[r]&\langle\delbar\del  a\rangle\\
	\langle a\rangle\ar[u]\ar[r]&\langle\del a\rangle\ar[u];
\end{tikzcd}
\]

{\bf dots}, namely double complexes of the shape $\langle a\rangle$; and {\bf zigzags} of even length $\leq 2(r-1)$.

\item[(4)] There is an equality
\[
h_{BC}(X)=\sum_{i=1}^{r-1}e_i(X)-(r-2)b(X),
\]
where $h_{BC}(X):=\sum_{p,q\in\Z}\dim H_{BC}^{p,q}(X)$ is the total Bott-Chern cohomology dimension, $e_i(X):=\sum_{p,q}\dim E_i^{p,q}(X)$ is the total dimension of the $i$-th page of the Fr\"olicher spectral sequence and $b(X):=\sum_{i\in\Z}=b_i(X)$ is the total Betti-number.
\end{enumerate}
\vspace{1ex}
A compact complex manifold X that satisfies any of these equivalent conditions is said to be a {\bf page-$(r-1)$-$\partial\bar\partial$-manifold}.

\end{The-Def}

\vspace{2ex}

For example, when they are of lengths $2$ and $4$, the zigzags mentioned under $(3)$ in the above statement are of the following types: \[
\begin{tikzcd}
	\langle a\rangle \ar[r]&\langle \del a\rangle,
\end{tikzcd}\qquad
\begin{tikzcd}
	\langle\bar\partial a\rangle\\
	\langle a\rangle\ar[u],
\end{tikzcd}\qquad
\begin{tikzcd}
	\langle a_1\rangle\ar[r]&\langle\partial a_1\rangle&\\
	&\langle a_2\rangle\ar[u]\ar[r]&\langle \partial a_2\rangle,
\end{tikzcd}\qquad
\begin{tikzcd}
\langle \bar\partial a_1\rangle&\\
\langle a_1\rangle\ar[r]\ar[u]&\langle \bar\partial a_2\rangle\\
&\langle a_2\rangle\ar[u].
\end{tikzcd}
\]

In all these diagrams, $a$ and $a_i$ are non-zero elements of pure bidegrees, all drawn arrows are supposed to be isomorphisms and all omitted arrows are zero.

Note that the pure-Hodge condition of the de Rham cohomology imposes topological restrictions on the underlying manifold: as in the K\"ahler case, the odd-degree Betti numbers $b_{2k+1}(X)$ of page-$r$-$\partial\bar\partial$-manifolds have to be even.
\vspace{1ex}

Condition $(4)$ is actually the equality case of the following general inequality, proved in section \ref{section:page-r}.

\begin{The}\label{The:page-1-ddbar_numerical-char} Let $X$ be a compact complex manifold with $\mbox{dim}_\C X=n$. Then, for any $r\in \mathbb{N}^\star$, the following inequality holds
	\[
	h_{BC}(X)\geq \sum_{i=1}^{r-1} e_i(X)-(r-2)b(X),
	\]	
	with equality if and only if $X$ is a page-$(r-1)$-$\partial\bar\partial$-manifold.
\end{The}
Note that one can rewrite this inequality to obtain a lower bound for the total Betti number (a topological quantity) in terms of analytic invariants.

The statement for $r=1$ (i.e. $h_{BC}(X)\geq b(X)$) and for $r=2$ (i.e. $h_{BC}(X)\geq h_{\delbar}(X)$) were known (see [8]) but the characterisation of the equality is new for $r=2$. For $r\geq 3$, both the inequality and the characterisation of the equality are new.\\

In section \ref{subsection: examples}, we provide {\bf examples} of {\it primary} (i.e. not derived from others of their kind) {\bf page-$r$-$\partial\bar\partial$-manifolds} that are not page-$0$-$\partial\bar\partial$-manifolds (i.e. not $\partial\bar\partial$-manifolds in the usual sense).
We produce three different classes of such examples with $r=1$: 
 
 \begin{enumerate}
\item[(1)] all the {\bf complex parallelisable nilmanifolds} (see Theorem \ref{The:nilmanifolds_c-par}). These are not $\partial\bar\partial$-manifolds unless they are complex tori. This class of manifolds includes the complex $3$-dimensional Iwasawa manifold, but is far wider;
\item[(2)] two families of \textbf{nilmanifolds with abelian complex structures} with members of arbitrarily high dimensions. (See Theorem \ref{AbelianNilmanifolds} and Proposition \ref{Prop:clasif-nilmanifolds}.) In some sense, these form the opposite of the above class (1) among nilmanifolds (see Remarks \ref{Dolb-A1-A2} and \ref{sGG-A1-A2}); 
\item[(3)] the {\bf Nakamura solvmanifolds} (which are not nilmanifolds) considered in [29] and [4]. (See Corollary \ref{Cor:A-K_Nakamura}.)
\end{enumerate}

In the subsequent work [20] by H. Kasuya and the second-named author, examples (1) and (3) were generalised; in particular, all complex parallelisable solvmanifolds were shown to be page-$1$-$\del\delbar$-manifolds.\\

In section \ref{section:geom-prop_page-r-ddbar}, we study the behaviour of page-$r$-$\partial\bar\partial$-manifolds under standard geometric operations. In particular, we obtain construction methods for new examples from given ones. These include:

\vspace{1ex}

(4)\, {\bf products} of page-$r_i$-$\partial\bar\partial$-manifolds, with possibly different $r_i$'s;

\vspace{1ex}

(5)\, {\bf blow-ups} of page-$r_1$-$\partial\bar\partial$-submanifolds of page-$r_2$-$\partial\bar\partial$-manifolds, with possibly different $r_i$'s;

\vspace{1ex}

(6)\, the {\bf projectivised bundle} $\Proj(\mathcal{V})$ of any holomorphic vector bundle $\mathcal{V}$ on a page-$r$-$\partial\bar\partial$-manifold;

\vspace{1ex}

(7)\, {\bf small deformations} of a page-$1$-$\partial\bar\partial$-manifold with fixed Hodge numbers.

\section{Page-$r$-$\partial\bar\partial$-manifolds}\label{section:page-r-dd-bar} In this section, we give the main definitions and some of the basic properties of the new class of manifolds that we introduce herein. Unless otherwise stated, $X$ will stand for an $n$-dimensional compact complex manifold. We refer to [17] for the setup of the Fr\"olicher spectral sequence and [13] for a hands-on description of the groups appearing on the higher pages.

\subsection{Preliminaries}\label{subsection:preliminaries} We start by recalling some well-known facts in order to fix the setup and to spell out the relationship between the Fr\"olicher spectral sequence and the (filtered) de Rham cohomology when the latter is {\it pure} in a sense that will be specified.\\

For all non-negative integers $k\leq 2n$ and $p\leq\min\{k,\,n\}$, it is standard to put \begin{equation}\label{eqn:def_filtration-forms}{\cal F}^p C^\infty_k(X,\,\C):=\bigoplus\limits_{i\geq p}C^\infty_{i,\,k-i}(X)\subset C^\infty_k(X,\,\C)\end{equation}

\noindent and get a filtration of $C^\infty_k(X,\,\C)$ for every $k$: \begin{equation}\label{eqn:filtration-forms}\{0\}\subset\dots\subset {\cal F}^{p+1}C^\infty_k(X,\,\C)\subset {\cal F}^pC^\infty_k(X,\,\C)\subset\dots\subset C^\infty_k(X,\,\C).\end{equation}

It is equally standard to put $$F^pH^k_{dR}(X,\,\C):=\frac{{\cal F}^p C^\infty_k(X,\,\C)\cap\ker d}{{\cal F}^p C^\infty_k(X,\,\C)\cap\mbox{im}\,d}\subset H^k_{dR}(X,\,\C),$$

\noindent the subspace of de Rham cohomology classes of degree $k$ that are representable by forms in ${\cal F}^p C^\infty_k(X,\,\C)$, to get a filtration of $H^k_{dR}(X,\,\C)$ for every $k$:

\begin{equation}\label{eqn:Hk-filtration}\{0\}\subset\dots\subset F^{p+1}H^k_{dR}(X,\,\C)\subset F^pH^k_{dR}(X,\,\C)\subset\dots\subset H^k_{dR}(X,\,\C).\end{equation}

\vspace{2ex}

Let us now recall the following standard result (see e.g. [17, Lemma 2]).

\begin{The}\label{The:standard_E_infty-filtration} Let $X$ be an $n$-dimensional compact complex manifold. For every $p,q\in\{0,\dots , n\}$, the vector space $E_\infty^{p,\,q}(X)$ of type $(p,\,q)$ on the degenerating page of the Fr\"olicher spectral sequence of $X$ is {\bf naturally isomorphic} to the graded module associated with the filtration (\ref{eqn:Hk-filtration}): $$E_\infty^{p,\,q}(X) \simeq G_pH^{p+q}_{dR}(X,\,\C):=\frac{F^pH^{p+q}_{dR}(X,\,\C)}{F^{p+1}H^{p+q}_{dR}(X,\,\C)},$$

  \noindent where the isomorphism $G_pH^{p+q}_{dR}(X,\,\C)\simeq E_\infty^{p,\,q}(X)$ is induced by the projection $$F^pH^{p+q}_{dR}(X,\,\C)\ni\bigg\{\sum\limits_{i\geq p}u^{i,\,p+q-i}\bigg\}_{dR}\mapsto\{u^{p,\,q}\}_{E_\infty}\in E_\infty^{p,\,q}(X).$$

\end{The}

\vspace{2ex}

The following statement is immediate to prove.

\begin{Lem} The following relations hold:  \begin{eqnarray}\label{eqn:forms_complementarity}C^\infty_k(X,\,\C) & = & {\cal F}^p C^\infty_k(X,\,\C)\oplus\overline{{\cal F}^{k-p+1} C^\infty_k(X,\,\C)} \hspace{6ex} \mbox{for all}  \hspace{2ex} 0\leq p\leq \min\{k,\,n\}; \\
\label{eqn:forms_intersection} C^\infty_{p,\,q}(X) & = & {\cal F}^p C^\infty_k(X,\,\C)\cap\overline{{\cal F}^q C^\infty_k(X,\,\C)} \hspace{6ex} \mbox{for all}  \hspace{2ex} p,q \hspace{1ex} \mbox{such that}  \hspace{1ex} p+q=k.\end{eqnarray}

\end{Lem}

\vspace{2ex}

On the other hand, for all $p,q\in\{0,\dots , n\}$, let us consider the following space of de Rham cohomology classes of degree $p+q$ that are representable by pure-type $(p,\,q)$-forms: \[H^{p,\,q}_{dR}(X):=\bigg\{\mathfrak{c}\in H^{p+q}_{dR}(X,\,\C)\,\mid\,\exists\alpha\in C^\infty_{p,\,q}(X)\text{ s.t. }[\alpha]=\mathfrak{c}\bigg\}\subset H^{p+q}_{dR}(X,\,\C).\]

This definition makes it obvious that the analogue of the Hodge symmetry for the spaces $H^{p,\,q}_{dR}(X)$ always holds. In other words, the conjugation induces an isomorphism

\begin{equation}\label{Def:DR_Hodge-sym}H^{p,\,q}_{dR}(X)\in\{\alpha\}_{dR}\mapsto \overline{\{\bar\alpha\}}_{dR}\in\overline{H^{q,\,p}_{dR}(X)}  \hspace{6ex} \mbox{for all}  \hspace{2ex} 0\leq p,q\leq n.\end{equation}

The following analogue in cohomology of identity (\ref{eqn:forms_intersection}), resp. of one of the inclusions defining the filtration (\ref{eqn:def_filtration-forms}) of $C^\infty_k(X,\,\C)$, can be immediately proved to hold.

\begin{Lem} The following relations hold:  \begin{eqnarray}\label{eqn:cohomology_complementarity_1st-inclusion}H^{p,\,q}_{dR}(X) = F^pH^k_{dR}(X,\,\C)\cap\overline{F^qH^k_{dR}(X,\,\C)} & & \hspace{6ex} \mbox{for all}  \hspace{2ex} p,q \hspace{1ex} \mbox{such that}  \hspace{1ex} p+q=k; \\
\label{eqn:purity-filtration_one-inclusion}H^{i,\,k-i}_{dR}(X) \subset F^pH^k_{dR}(X,\,\C) & & \hspace{6ex} \mbox{for all}  \hspace{2ex} i\geq p \hspace{1ex} \mbox{and all} \hspace{2ex} p\leq k.   \end{eqnarray}  

\end{Lem}

\noindent {\it Proof.} Everything is obvious, except perhaps the inclusion ``$\supset$'' in (\ref{eqn:cohomology_complementarity_1st-inclusion}) which can be proved as follows. Let $\{\alpha\}_{dR} = \{\beta\}_{dR}\in F^pH^k_{dR}(X,\,\C)\cap\overline{F^qH^k_{dR}(X,\,\C)}$ with $\alpha=\sum_{i\geq p}\alpha^{i,\,k-i}\in{\cal F}^pC^\infty_k(X,\,\C)$ and $\beta=\sum_{s\leq p}\beta^{s,\,k-s}\in\overline{{\cal F}^qC^\infty_k(X,\,\C)}$. Since $\alpha$ and $\beta$ are de Rham-cohomologous, there exists a form $\sigma\in C^\infty_{k-1}(X,\,\C)$ such that $\alpha-\beta=d\sigma$. This identity implies, after equating the terms with a holomorphic degree $>p$ on either side, the second identity below: $$\alpha - \alpha^{p,\,q}  =\sum_{i> p}\alpha^{i,\,k-i} = d(\sum_{j\geq p}\sigma^{j,\,k-1-j}) - \bar\partial\sigma^{p,\,q-1},$$

\noindent which, in turn, implies that $\{\alpha\}_{dR} = \{\alpha^{p,\,q} - \bar\partial\sigma^{p,\,q-1}\}_{dR}$. Since $\alpha^{p,\,q} - \bar\partial\sigma^{p,\,q-1}$ is a $(p,\,q)$-form, we get $\{\alpha\}_{dR}\in H^{p,\,q}_{dR}(X)$ and we are done. \hfill $\Box$

\vspace{3ex}

Note that, with no assumption on $X$, the subspaces $H^{i,\,k-i}_{dR}(X)$ may have non-zero mutual intersections inside $H^k_{dR}(X,\,\C)$, i.e. they may not sit in a direct sum. Similarly, they may not fill out the whole space (i.e. their linear span could be a proper subspace). If they do, i.e. if 
\[
H^k_{dR}(X,\,\C) = \bigoplus\limits_{p+q=k}H^{p,\,q}_{dR}(X),
\] then $H^k_{dR}(X,\, \C)$ is said to carry a pure Hodge structure (induced from the Hodge filtration), [14]. We also recall from [14] that purity in degree $k$ is equivalent to the filtrations $F$ and $\bar{F}$ on $H_{dR}^k(X,\,\C)$ being $k$-opposed, which means that the natural map $F^pH_{dR}^k(X,\,\C)\oplus \overline{F^qH_{dR}^k(X,\,\C)}\longrightarrow H_{dR}^k(X,\,\C)$ is an isomorphism whenever $p+q=k+1$, or equivalently that $\operatorname{gr}_{F}^p\operatorname{gr}_{\bar F}^q H_{dR}^k(X,\,\C)=0$ whenever $p+q\neq k$.
We now introduce the following shorthand terminology.

\begin{Def}\label{Def:purity} Let $X$ be an $n$-dimensional compact complex manifold. The de Rham cohomology of $X$ is said to be {\bf pure} if the Hodge filtration induces a pure Hodge structure in every degree, i.e. if $$H^k_{dR}(X,\,\C) = \bigoplus\limits_{p+q=k}H^{p,\,q}_{dR}(X)  \hspace{6ex} \mbox{for all}  \hspace{2ex} k\in\{0,\dots , 2n\}.$$
\end{Def}

\begin{Terminology}\label{rem-terminology}

 Some authors call the property of the Hodge filtration inducing a pure Hodge structure in degree $k$ {\bf complex-$\mathcal{C}^\infty$-pure-and-full} in degree $k$ (cf. [25]). It was remarked in [7] that the {\bf complex-$\mathcal{C}^\infty$-full} property {in degree $k$} (i.e. the sum of the $H^{p,\,q}_{dR}(X)$'s is not necessarily direct but it fills out $H^k_{dR}(X,\,\C)$) implies the {\bf complex-$\mathcal{C}^\infty$-pure} property {in degree $(2n-k)$} (i.e. the sum of the $H^{p,\,q}_{dR}(X)$'s is direct but it may not fill out $H^{2n-k}_{dR}(X,\,\C)$). We will show further down that the converse is also true, i.e. the {\bf complex-$\mathcal{C}^\infty$-pure} property {in degree $k$} implies the {\bf complex-$\mathcal{C}^\infty$-full} property {in degree $(2n-k)$}. Therefore, compact complex manifolds satisfying either the {\bf complex-$\mathcal{C}^\infty$-full} property or the {\bf complex-$\mathcal{C}^\infty$-pure} property {in every degree $k$} are of {\bf pure} de Rham cohomology in the sense of our Definition~\ref{Def:purity}.

\end{Terminology}

\begin{Prop}\label{Prop:purity-filtration} Suppose $X$ is an $n$-dimensional compact complex manifold. Then its de Rham cohomology is {\bf pure} if and only if 
	
	\begin{eqnarray}\label{eqn:purity-filtration}F^pH^k_{dR}(X,\,\C) = \bigoplus\limits_{i\geq p}H^{i,\,k-i}_{dR}(X) \hspace{6ex} \mbox{for all}  \hspace{2ex} p\leq k.\end{eqnarray}

  \noindent In particular, for pure $X$, the spaces $E_\infty^{p,\,q}(X)$ in the Fr\"olicher spectral sequence of $X$ are given by 
  
  \begin{equation}\label{eqn:E_infty_pure-DR}E_\infty^{p,\,q}(X) \simeq H^{p,\,q}_{dR}(X) \hspace{6ex} \mbox{for all}  \hspace{2ex} p,q\in\{0,\dots , n\},\end{equation}

  \noindent where $\simeq$ stands for the natural isomorphism induced by the identity.

\end{Prop}

\noindent {\it Proof.} Equation \eqref{eqn:purity-filtration} for all $p=k$ is just the definition of purity, so it remains to prove the converse. 

Assume $X$ is pure. Inclusion ``$\supset$'' in (\ref{eqn:purity-filtration}) follows at once from (\ref{eqn:purity-filtration_one-inclusion}) and from the de Rham purity assumption. To prove inclusion ``$\subset$'' in (\ref{eqn:purity-filtration}), let $\{\alpha\}_{dR}\in F^pH^k_{dR}(X,\,\C)$ with $\alpha=\sum\limits_{r\geq p}\alpha^{r,\,k-r}\in\ker d$. Since $F^pH^k_{dR}(X,\,\C)\subset H^k_{dR}(X,\,\C) = \oplus_{p+q=k}H^{p,\,q}_{dR}(X)$ (the last identity being due to the purity assumption), there exist pure-type $d$-closed forms $\beta^{r,\,k-r}$ such that $\{\alpha\}_{dR} = \{\sum_{0\leq r\leq k}\beta^{r,\,k-r}\}_{dR}$. Hence, there exists a $(k-1)$-form $\sigma$ such that $\alpha-\sum_{0\leq r\leq k}\beta^{r,\,k-r} = -d\sigma$, which amounts to $$\alpha^{r,\,k-r}-\beta^{r,\,k-r}+\partial\sigma^{r-1,\,k-r}+\bar\partial\sigma^{r,\,k-r-1}=0, \hspace{6ex} r\in\{0,\dots , k\},$$

\noindent with the understanding that $\alpha^{r,\,k-r}=0$ whenever $r<p$.

Therefore, $\beta^{r,\,k-r} = \partial\sigma^{r-1,\,k-r}+\bar\partial\sigma^{r,\,k-r-1}$ whenever $r<p$. Since every $\beta^{r,\,k-r}$ is $d$-closed (hence also $\partial$- and $\bar\partial$-closed), we infer that $\sigma^{r-1,\,k-r}$ and $\sigma^{r,\,k-r-1}$ are $\partial\bar\partial$-closed for every $r<p$. Hence \begin{equation}\label{eqn:purity-filtration_1}\begin{array}{rcl}

\sum\limits_{r=0}^k\beta^{r,\,k-r} - d(\sum\limits_{r<p}\sigma^{r,\,k-r-1}) & = & \sum\limits_{r<p}(\beta^{r,\,k-r} - \partial\sigma^{r-1,\,k-r} - \bar\partial\sigma^{r,\,k-r-1}) - \partial\sigma^{p-1,\,k-p} + \sum\limits_{r\geq p}^k\beta^{r,\,k-r}\\
   & = & \sum\limits_{r\geq p}\beta^{r,\,k-r} - \partial\sigma^{p-1,\,k-p}.\end{array}
\end{equation}

\noindent Note that from the identity $\beta^{p-1,\,k-p+1} = \partial\sigma^{p-2,\,k-p+1}+\bar\partial\sigma^{p-1,\,k-p}$ and the $d$-closedness of $\beta^{p-1,\,k-p+1}$ we infer that $\partial\sigma^{p-1,\,k-p}\in\ker d$.

Thus, (\ref{eqn:purity-filtration_1}) shows that the $k$-form $\sum_{r\geq p}\beta^{r,\,k-r} - \partial\sigma^{p-1,\,k-p}\in{\cal F}^pC^\infty_k(X,\,\C)$, whose all pure-type components are $d$-closed, is de Rham-cohomologous to $\sum_{0\leq r\leq k}\beta^{r,\,k-r}$, hence to $\alpha$. Consequently, we have $$\{\alpha\}_{dR} = \bigg\{\sum\limits_{r\geq p}\beta^{r,\,k-r} - \partial\sigma^{p-1,\,k-p}\bigg\}_{dR}\in\bigoplus\limits_{i\geq p}H^{i,\,k-i}_{dR}(X).$$

\noindent The proof of (\ref{eqn:purity-filtration}) is complete.

Identity (\ref{eqn:E_infty_pure-DR}) follows at once from (\ref{eqn:purity-filtration}) and from Theorem \ref{The:standard_E_infty-filtration}.   \hfill $\Box$

\subsection{Definition of page-$r$-$\partial\bar\partial$-manifolds}\label{subsection:page-r-dd-bar_def} Recall that $X$ is a fixed $n$-dimensional compact complex manifold and $E_r^{p,\,q}(X)$ stands for the space of bidegree $(p,\,q)$ on the $r$-th page of the Fr\"olicher spectral sequence of $X$.

\begin{Def}\label{Def:id_isom_decomp} Fix $r\in\N^\star$ and $k\in\{0,\dots , 2n\}$. We say that the {\bf identity induces an isomorphism} between $\oplus_{p+q=k}E_r^{p,\,q}(X)$ and $H^k_{dR}(X,\,\C)$ if the following two conditions are satisfied:

\begin{enumerate}
\item[(1)] for every bidegree $(p,\,q)$ with $p+q=k$, every class $\{\alpha^{p,\,q}\}_{E_r}\in E_r^{p,\,q}(X)$ contains a {\bf $d$-closed representative of pure type} $\alpha^{p,\,q}\in C^\infty_{p,\,q}(X)$ ;

\item[(2)] the linear map $$\bigoplus_{p+q=k}E_r^{p,\,q}(X)\ni\sum\limits_{p+q=k}\{\alpha^{p,\,q}\}_{E_r}\mapsto\bigg\{\sum\limits_{p+q=k}\alpha^{p,\,q}\bigg\}_{dR}\in H^k_{dR}(X,\,\C),$$

  \noindent defined using only $d$-closed pure-type representatives $\alpha^{p,\,q}$ of the classes $\{\alpha^{p,\,q}\}_{E_r}$, whose existence is guaranteed by condition $(1)$, is {\bf well-defined}  and {\bf bijective}. Here, well-defined means that it does not depend on the choice of the $d$-closed representatives.

\end{enumerate}

 Moreover, if, for a fixed $r\in\N^\star$, the identity induces an isomorphism $\oplus_{p+q=k}E_r^{p,\,q}(X)\simeq H^k_{dR}(X,\,\C)$ for every $k\in\{0,\dots , 2n\}$, we say that the manifold $X$ has the {\bf $E_r$-Hodge Decomposition} property.

\end{Def}

\vspace{2ex}

Note that whenever the identity induces a well-defined (not necessarily injective) linear map $E_r^{p,\,q}(X)\longrightarrow H^k_{dR}(X,\,\C)$, the image of this map is $H^{p,\,q}_{dR}(X)$. Indeed, one inclusion is obvious. The reverse inclusion follows from the observation that any $d$-closed $(p,\,q)$-form defines an $E_r$-cohomology class (i.e. it is $E_r$-closed in the terminology of [32]). Further note that whenever $X$ has the $E_r$-Hodge Decomposition property, the Fr\"olicher spectral sequence of $X$ degenerates at $E_r$ (at the latest).

 \begin{Def}\label{Def:conj_isom_sym} Fix $r\in\N^\star$ and $p,q\in\{0,\dots , n\}$. We say that the {\bf conjugation induces an isomorphism} between $E_r^{p,\,q}(X)$ and the conjugate of $E_r^{q,\,p}(X)$ if the following two conditions are satisfied:
\begin{enumerate}
\item[(1)]
every class $\{\alpha^{p,\,q}\}_{E_r}\in E_r^{p,\,q}(X)$ contains a {\bf $d$-closed representative of pure type} $\alpha^{p,\,q}\in C^\infty_{p,\,q}(X)$;

\item[(2)] the linear map $$E_r^{p,\,q}(X)\ni\{\alpha^{p,\,q}\}_{E_r}\mapsto\overline{\{\overline{\alpha^{p,\,q}}\}_{E_r}}\in \overline{E_r^{q,\,p}(X)}$$

  \noindent is {\bf well-defined} (in the sense that it does not depend on the choice of $d$-closed representative $\alpha^{p,\,q}$ of the class $\{\alpha^{p,\,q}\}_{E_r}$) and {\bf bijective}.

\end{enumerate}

  Moreover, if, for a fixed $r\in\N^\star$, the conjugation induces an isomorphism $E_r^{p,\,q}(X)\simeq\overline{E_r^{q,\,p}(X)}$ for every $p,q\in\{0,\dots , n\}$, we say that the manifold $X$ has the {\bf $E_r$-Hodge Symmetry} property.

\end{Def}

 \vspace{2ex}

 We shall now see that the $E_r$-Hodge Decomposition property implies the $E_r$-Hodge Symmetry property. This follows from the following characterisation of the former property.

 \begin{The}\label{The:E_r_Hodge-decomp_pureDR} Let $X$ be a compact complex manifold with $\mbox{dim}_\C X=n$. Fix an arbitrary $r\in\N^\star$. Then, the following two conditions are equivalent:

\begin{enumerate}
	\item[(1)] $X$ has the {\bf $E_r$-Hodge Decomposition} property;
\item[(2)] the Fr\"olicher spectral sequence of $X$ {\bf degenerates at $E_r$} (we will denote this by $E_r(X)=E_\infty(X)$) and the de Rham cohomology of $X$ is {\bf pure}.

\end{enumerate}

 \end{The}

 \noindent {\it Proof.} $(1)\implies(2)$\, We have already noticed that the $E_r$-Hodge Decomposition property implies $E_r(X)=E_\infty(X)$ and that the image of each $E_r^{p,\,q}(X)$ in $H^{p+q}_{dR}(X,\,\C)$ under the map induced by the identity is $H^{p,\,q}_{dR}(X)$. We get $(2)$.

 \vspace{1ex}
 $(2)\implies(1)$\, Since the de Rham cohomology of $X$ is supposed pure, we know from Proposition \ref{Prop:purity-filtration} that $E_\infty^{p,\,q}(X)\simeq H^{p,\,q}_{dR}(X)$ (isomorphism induced by the identity) for all bidegrees $(p,\,q)$. On the other hand, $E_\infty^{p,\,q}(X) = E_r^{p,\,q}(X)$ for all bidegrees $(p,\,q)$ since we are assuming that $E_r(X)=E_\infty(X)$. Combined with the de Rham purity assumption, these facts imply that $X$ has the $E_r$-Hodge Decomposition property. \hfill $\Box$

\begin{Def}\label{Def:E_r_Hodge-decomp_pureDR}
	A compact complex manifold $X$ that satisfies the equivalent conditions $(1)$ and $(2)$ of Theorem \ref{The:E_r_Hodge-decomp_pureDR} is said to be a {\bf page-$(r-1)$-$\partial\bar\partial$-manifold}.
\end{Def}

 \begin{Cor}\label{Cor:E_r_decomp-sym} Any page-$(r-1)$-$\partial\bar\partial$-manifold has the $E_r$-Hodge Symmetry property.

\end{Cor}

 \noindent {\it Proof.}  We have already noticed in (\ref{Def:DR_Hodge-sym}) that the conjugation (trivially) induces an isomorphism between any space $H_{dR}^{p,\,q}(X)$ and the conjugate of $H_{dR}^{q,\,p}(X)$. Meanwhile, we have seen that the page-$(r-1)$-$\partial\bar\partial$-assumption implies that the identity induces an isomorphism between any space $E_r^{p,\,q}(X)$ and $H_{dR}^{p,\,q}(X)$. Hence, the conjugation induces an isomorphism between any space $E_r^{p,\,q}(X)$ and the conjugate of $E_r^{q,\,p}(X)$.  \hfill $\Box$

\vspace{2ex}

Another obvious consequence of $(2)$ of Theorem \ref{The:E_r_Hodge-decomp_pureDR} and Definition \ref{Def:E_r_Hodge-decomp_pureDR} is that the page-$r$-$\partial\bar\partial$-property becomes weaker and weaker as $r$ increases.

\begin{Cor}\label{Cor:increasing-r_page-r-ddbar} Let $X$ be a compact complex manifold. Then, for every $r\in\N^\star$, the following implication holds:

\vspace{2ex}

\hspace{6ex} $X$ is a page-$r$-$\partial\bar\partial$-manifold $\implies$ $X$ is a page-$(r+1)$-$\partial\bar\partial$-manifold.

\end{Cor}

Indeed, the purity of the de Rham cohomology is independent of $r$, while the property $E_r(X)=E_\infty(X)$ obviously implies $E_{r+1}(X)=E_\infty(X)$ for every $r\in\N$.

\subsection{Characterisation in terms of squares and zigzags}\label{subsection:squares-zigzags}

The goal of this section is to relate the page-$r$-$\del\delbar$-property to structural results about double complexes. This degree of generality has the advantage of emphasising which aspects of the theory are purely algebraic. Even if one is only interested in the complex $A_X:=(C^\infty_{p,\,q}(X),\del,\delbar)$ of $\C$-valued forms on a complex manifold $X$, in the more general setting one can consider certain finite-dimensional subcomplexes on an equal footing.

Specifically, by double complexes we mean bigraded vector spaces $A=\bigoplus_{p,q\in\Z} A^{p,q}$ with endomorphisms $\del_1,\del_2$ of bidegrees $(1,0)$, resp. $(0,1)$, satisfying $d^2=0$ for $d:=\del_1+\del_2$. We do not require $A$ to be finite-dimensional. We will always assume our double complexes to be bounded, i.e. $A^{p,q}=0$ for all but finitely many $(p,q)\in\Z^2$.

There are now two Fr\"olicher-style spectral sequences, starting from column, i.e. ($\partial_2$-), resp. row, i.e. ($\partial_1$-), cohomology and converging to the total (de Rham) cohomology of $(A,\,d)$. We denote them by

\[
_iE^{p,q}_r(A)\Longrightarrow (H_{dR}^{p+q}(A),F_i)\quad i=1,2.
\]

In the case $A=A_X$, the case $i=1$ is the Fr\"olicher spectral sequence and $i=2$ its conjugate.\\

The following is a minor extension to general double complexes of the definition (based on its second characterisation) of the page-$(r-1)$-$\del\delbar$-property of manifolds. The equivalence of the two conditions is seen just as before.

\begin{Def}
A double complex $A$ is said to satisfy the {\bf page-$(r-1)$-$\del_1\del_2$-property} if one (hence both) of the following equivalent conditions hold:
\begin{enumerate}
	\item[(1)] Both Fr\"olicher spectral sequences degenerate at page $r$ and the de Rham cohomology is pure.
	\item[(2)] For $i=1,2$, every $_iE_r^{p,q}(A)$-class contains a $d$-closed representative and the corresponding map \[
	\bigoplus_{p+q=k}{}_iE_r^{p,q}(A)\longrightarrow H_{dR}^k(A)
	\] induced by the identity is well-defined and bijective.
\end{enumerate}
\end{Def}

The following observation will motivate the subsequent considerations.

\begin{Obs}\label{direct sums and ddbar}

The Fr\"olicher spectral sequences, as well as $H_{dR}, H_{A}$ and $H_{BC}$, are compatible with direct sums. In particular, a sum $A=B\oplus C$ satisfies the page-$r$-$\del_1\del_2$-property if and only if $B$ and $C$ do.

\end{Obs}

Recall that a (nonzero) double complex $A$ is called {\bf indecomposable} if there exists no nontrivial decomposition $A=B\oplus C$ into subcomplexes $B,C$.

\begin{The}([21, 38])\label{The:Decomposition} For every bounded double complex over a field $K$, there exists an {\bf isomorphism} 

\[
A\cong \bigoplus_{C} C^{\oplus \mult_{C}(A)},
\] where $C$ runs over a set of representatives for the isomorphism classes of bounded {\bf indecomposable} double complexes and $\mult_{C}(A)$ are (not necessarily finite) cardinal numbers uniquely determined by $A$.

Moreover, each bounded {\bf indecomposable} double complex is finite dimensional and isomorphic to a complex of one of the following types:

\begin{enumerate}

\item[(1)] {\bf square}: a double complex generated by a single pure-$(p,q)$-type element $a$ in a given bidegree with no further relations (i.e. $\del_1a\neq 0\neq \del_2a, \del_1\del_2a\neq 0$):

\[
\begin{tikzcd}
\langle\del_2 a\rangle\ar[r]&\langle\del_2\del_1  a\rangle\\
\langle a\rangle\ar[u]\ar[r]&\langle\del_1 a\rangle\ar[u].
\end{tikzcd}
\]
\item[(2)] {\bf even-length zigzag of type $1$ and length $2l$}. This is a complex generated by elements $a_1,...a_l$ and their differentials such that $\del_2 a_1=0$ and $\del_1 a_1=-\del_2 a_2$, $\del_1 a_2=-\del_2 a_3$, ..., $\del_1 a_{l-1}=-\del_2 {a_l}$ and no further relations (i.e. $\del_1 a_{l}\neq 0$ for all $i=1,...,l$). Drawing only nonzero arrows, it has the shape:

\[
\begin{tikzcd}
& \langle a_1\rangle \arrow{r} & \langle \del_1 a_1 \rangle \\
& & \langle a_2 \rangle \arrow{u} \arrow{r} & \arrow[draw=none]{rd}{\cdots} \\
& & & & { } \\
& & & & \langle a_l \rangle \arrow{u} \arrow{r} & \langle \del_1 a_l \rangle.
\end{tikzcd}
\]
Here, as in all the following examples, the length of a zigzag is the number of its vertices.

\item[(3)] {\bf even-length zigzag of type $2$ and length $2l$}. This is a complex generated by elements $a_1,...,a_l$ and their differentials, such that $\del_1a_1=-\del_2 a_2$, $\del_1 a_2=-\del_2 a_3$, ..., $\del_1a_{l-1}=-\del_2a_l$, $\del_1 a_l=0$ and no further relations. It is of the shape:

\[
\begin{tikzcd}
&\langle \del_2 a_1\rangle\\
& \langle a_1\rangle \arrow{r}\arrow{u} & \langle \del_1 a_1 \rangle \\
& & \langle a_2 \rangle \arrow{u} \arrow{r} & \arrow[draw=none]{rd}{\cdots} \\
& & & & { } \\
& & & & \langle a_l \rangle. \arrow{u}
\end{tikzcd}
\]

\item[(4)] {\bf odd-length zigzag of type $M$ and length $2l+1$}. This is a complex generated by elements $a_1,...,a_{l+1}$ with the only relations $\del_1 a_i=-\del_2a_{i+1}$, $\del_2 a_1= 0$ and $\del_1 a_{l+1}=0$. It has the shape:

\[
\begin{tikzcd}
& \langle a_1\rangle \arrow{r} & \langle \del_1 a_1 \rangle \\
& & \langle a_2 \rangle \arrow{u} \arrow{r} & \arrow[draw=none]{rd}{\cdots} \\
& & & & { } \\
& & & & \langle a_{l+1} \rangle \arrow{u}.&
\end{tikzcd}
\]

The special case where $l=0$ is also called a {\bf dot}.

\item[(5)] {\bf odd-length zigzag of type $L$ and length $2l+1$ ($l>0$)}. This is a complex generated by elements $a_1,...,a_l$ with the only relations  $\del_1 a_i=-\del_2 a_{i+1}$ (i.e. $\del_2 a_i\neq 0\neq\del_1 a_i$ for all $i$). It has the shape:

\[
\begin{tikzcd}
& \langle \del_2 a_1\rangle &&&&\\
& \langle a_1\rangle \arrow{r}\arrow{u} & \langle \del_1 a_1 \rangle\arrow[draw=none]{rd}{\cdots}& \\
& & &{} & \langle \del_2 a_{l}\rangle \\
& & & & \langle a_l \rangle \arrow{u} \arrow{r} & \langle \del_1 a_l \rangle.
\end{tikzcd}
\] 

\end{enumerate}

\end{The}

\vspace{2ex}

Illustrating this result, we point out the following

\begin{Ex} Any bounded complex of $K$-vector spaces $(V^\Cdot,\,\delta)$ is a direct sum of complexes of the following two forms:

\[
...\longrightarrow 0\longrightarrow K^{\oplus n}\longrightarrow 0 \longrightarrow ...\quad\text{ and }\quad ...\longrightarrow 0\longrightarrow K^{\oplus m}\overset{\cong}{\longrightarrow}K^{\oplus m}\longrightarrow 0\longrightarrow ...
\]

This can easily be seen directly by picking succesive complements of $\im\delta\subseteq \ker\delta\subseteq V^k$ in every degree.

A special case of Theorem \ref{The:Decomposition} is obtained when one considers $(V^\Cdot,\,\delta)$ as a double complex concentrated on a single  horizontal line, i.e. setting $V^{p,0}:=V^p$ and $V^{p,q}:=0$ for $q\neq 0$, $\del_1:=\delta$ and $\del_2:=0$. 

\end{Ex}

As an application of Theorem \ref{The:Decomposition}, we get the following result expressed in the language of this work.

\begin{The}\label{The: Characterisation zigzags} Let $A$ be a bounded double complex over a field $K$. The following statements are equivalent.\begin{enumerate} \item[(1)]
		$A$ satisfies the {\bf page-$r$-$\del_1\del_2$-property}.

\item[(2)] There exists an {\bf isomorphism} between $A$ and a {\bf direct sum} of {\bf squares}, {\bf even-length zigzags} of length $\leq 2r$ and {\bf odd-length zigzags} of length one (i.e. dots).

\end{enumerate}

\end{The}

\noindent {\it Proof.} This is a special case of [38, Thm C], which states in particular that the Fr\"olicher spectral sequences degenerate at page $r$ if and only if all even length zigzags of length $\geq 2r$ have multiplicity zero and that $H_{dR}^k(A)$ is pure of weight $k$ if and only if all odd-length zigzags of length $\geq 3$ have multiplicity zero.

For the reader's convenience and because we will need this type of reasoning again in the next section, we recall the proof in our case. By Observation \ref{direct sums and ddbar} and Theorem \ref{The:Decomposition}, it suffices to show that an {\it indecomposable} double complex satisfies the page-$r$-$\del_1\del_2$-property if and only if it is one of those listed under $(2)$ in the statement. We now run through the list given in Theorem \ref{The:Decomposition}.

\vspace{1ex}

$(1)$: For a square, both row and column cohomologies are zero. Therefore, the spectral sequences trivially degenerate and the total cohomology vanishes. In particular, the page-$r$-$\del_1\del_2$-property is trivially satisfied.

\vspace{1ex}

$(2)$ \& $(3)$: An even-length zigzag $Z$ of type $1$ and length $2l$ has vanishing row cohomology. Therefore, all terms in the second spectral sequence are identically zero and this spectral sequence degenerates for trivial reasons. On the other hand, the column cohomology is $2$-dimensional, with the two classes at the endpoints generating a one-dimensional space each. Since the total cohomology has to vanish (it is the limit of both spectral sequences), there must be a non-trivial differential at some page of $_1E(Z)$ and for space reasons it can only be at page $l$. Hence, $Z$ has the page-$r$-$\del_1\del_2$-property if and only if $l\leq r$. The case of an even length zigzag $Z$ of type $2$ is analogous, reversing the roles of row and column cohomology.

\vspace{1ex}

$(4)$ \& $(5)$: For an odd-length zigzag $Z$, both row and column cohomologies are one-dimensional, so there is no space for non-trivial differentials and both spectral sequences degenerate at the first page. Moreover, the de Rham cohomology is one-dimensional. If $Z$ is of type $M$, $H_{dR}(Z)$ is identified with the one-dimensional space generated by $a:=\sum_{i=1}^{l+1}a_i$, which is of pure-type only if $l=0$, namely if $Z$ is a dot. If $Z$ is of type $L$, it is generated by $[\del_2a_1]_{dR}=[-\del_1 a_1]_{dR}$, where the representatives live in different bidegrees. Therefore, $H_{dR}(Z)$ is not pure, so $Z$ does not have the page-$r$-$\del_1\del_2$-property. \hfill $\Box$

\begin{Rem} This theorem also gives a quick alternative proof to Prop. \ref{Prop:dd-bar_page_0_equiv} (equivalence of page-$0$-$\del_1\del_2$ with the usual $\del_1\del_2$-property). 

\end{Rem}

\noindent {\it Proof.} Indeed, the page-$0$-$\del_1\del_2$-property means that there is a decomposition of $A$ into squares and dots. Obviously, both satisfy the usual $\del_1\del_2$-property. Conversely, in any zigzag of length $\geq 2$ there is a closed element (`form') of pure type, which is $\del_1$- or $\del_2$-exact, but no non-zero element in a zigzag is $\del_1\del_2$-exact. Hence, if $A$ satisfies the usual $\del_1\del_2$-property, in any decomposition of $A$ into elementary complexes only squares and length-one zigzags can occur.  \hfill $\Box$

\begin{Def}\label{Def: E_r-iso}

A map $A\longrightarrow B$ of double complexes is an {\bf $E_r$-isomorphism} if $_iE_r(f)$ is an isomorphism for $i\in\{1,2\}$.

\end{Def}

One writes $A\simeq_r B$ if there exist such an $E_r$-isomorphism. The usefulness of this notion stems from the following statements.

\begin{Lem}\label{Lem: E_r isos in cohomology}([38, Prop. 12]) If $H$ is a linear functor from the category of double complexes to the category of vector spaces which maps squares and even-length zigzags of length $<2r$ to $0$, then $H(f)$ is an isomorphism for any $E_r$-isomorphism $f$.

\end{Lem}

\begin{Lem}\label{Lem: E_1 isos and zigzags}([38, Prop. 11]) For two double complexes $A,B$ one has $A\simeq_1 B$ if and only if `the same zigzags occur in $A$ and $B$', i.e. $\mult_Z(A)=\mult_Z(B)$ for all zigzags $Z$.

\end{Lem}

\begin{Ex} Examples of functors $H$ satisfying the hypotheses of Lemma \ref{Lem: E_r isos in cohomology} are provided by $H_{dR}$, $H_{BC}^{p,q}$, $E_i^{p,q}$ or $H_A^{p,q}$.

\end{Ex}

Thanks to its explicit description given above, one sees that an indecomposable double complex $C$ is determined up to isomorphism by its shape $S(C)=\{(p,q)\in \Z^2\mid C^{p,q}\neq 0\}$. By a slight abuse of notation, we will sometimes conveniently write $\mult_{S}(A)$  instead of $\mult_C(A)$ when $S=S(C)$.\\

We will need the following duality results in the special case $A=A_X$. They follow from the real structure and the Serre duality.

\begin{Lem}([38, Ch. 4])\label{Lem: dualites for dc} Let $A=A_X$ for a compact complex manifold $X$ and define the conjugate complex by $\bar{A}^{p,q}=\overline{A^{p,q}}$ and the dual complex $DA$ by $DA^{p,q}=\operatorname{Hom}(A^{n-p,n-q},\C)$, for all $p,q$.

 Then, conjugation $\omega\mapsto\overline{\omega}$ and integration $\omega\mapsto\int_X\omega\wedge\_$ define an isomorphism, resp. an $E_1$-isomorphism: $$A\cong \bar{A}, \hspace{3ex} \mbox{respectively} \hspace{3ex} A\rightarrow DA.$$

 In particular, the set of zigzags occuring in $A_X$ is symmetric under reflection along the diagonal and the anti-diagonal. More precisely, for any zigzag shape $S$, $\mult_S(A)=\mult_{rS}(A)=\mult_{dS}(A)$, where $rS=\{(p,q)\in\Z^2\mid (q,p)\in S\}$ and $dS:=\{(p,q)\in\Z^2\mid (n-p,n-q)\in S\}$. 

\end{Lem}

Recall that the complex-$\mathcal{C}^\infty$-pure and -full properties from Remark \ref{rem-terminology}. As a consequence of the above, we obtain

\begin{Prop}\label{Lem: pure<=>full} Fix arbitrary integers $0\leq k\leq 2n$. 

	A compact complex manifold $X$ of dimension $n$ satisfies the complex-$\mathcal{C}^\infty$-pure property in degree $k$ if and only if it satisfies the complex-$\mathcal{C}^\infty$-full property in degree $2n-k$
\end{Prop}

\begin{proof}
	Let $Z$ be a zigzag with $H_{dR}^k(Z)\neq 0$. The sum of the subspaces $H_{dR}^{p,q}(Z)$ with $p+q=k$ is not direct if and only if $Z$ is of odd length and of type $L$. Meanwhile, the sum of the subspaces $H_{dR}^{p,q}(Z)$ with $p+q=k$ is strictly contained in $H_{dR}^k(Z)$ if and only if $Z$ is of odd length $>1$ (i.e. not a dot) and of type $M$. We have seen both statements in the course of the proof of Theorem \ref{The: Characterisation zigzags}, see also [38, Prop. 6, Cor. 7]. Hence, $X$ is complex-$\mathcal{C}^\infty$-pure in degree $k$ if and only if $\mult_Z(A_X)=0$ for all odd zigzags $Z$ of type $L$ with $H_{dR}^k(Z)\neq 0$ and $X$ is complex $\mathcal{C}^\infty$-full in degree $k$ if and only if $\mult_Z(A_X)=0$ for all odd zigzags $Z$ of type $M$ and length $>1$ with $H_{dR}^k(Z)\neq 0$. The result now follows from Lemma \ref{Lem: dualites for dc} and Lemma \ref{Lem: E_1 isos and zigzags} since zigzags of type $L$ and those of type $M$ and length greater than $1$ are exchanged when forming the dual complex. \end{proof}

\begin{Cor}\label{Cor: pure and full} For a compact complex manifold $X$, the following statements are equivalent.

\vspace{1ex}

$(1)$\, $X$ satisfies the complex-$\mathcal{C}^\infty$-{\bf pure} property in all degrees;

\vspace{1ex}

$(2)$\, $X$ satisfies the complex-$\mathcal{C}^\infty$-{\bf full} property in all degrees;

\vspace{1ex}

$(3)$\,  The de Rham cohomology of $X$ is {\bf pure} (in the sense of Definition \ref{Def:purity}).

\end{Cor}

\section{Numerical characterisation of page-$r$-$\del\delbar$-manifolds and applications}\label{section:page-r}

Let $X$ be a compact connected complex manifold. Let $b(X)=\sum_{k\in \Z}b_k(X)$, $h_{BC}(X)=\sum_{p,q\in \Z} h_{BC}^{p,q}(X)$ and define $h_A(X)$, $h_{\del}(X)$ and $h_{\delbar}(X)$ analogously. Angella and Tomassini showed in [8] that there are inequalities: \begin{equation}h_{BC}(X)+h_{A}(X)\overset{(\ast)}{\geq} h_{\delbar}(X)+h_{\partial}(X)\overset{(\ast\ast)}{\geq} 2\, b(X)\label{ATin}\end{equation}

\noindent and that both of these inequalities are equalities if and only if $X$ is a $\partial\bar\partial$-manifold.\\

It is a standard fact about spectral sequences that equality in $(\ast\ast)$ is equivalent to the degeneration at $E_1$ of the Fr\"olicher spectral sequence (and its conjugate). One application of our methods is a generalisation of inequality $(\ast)$ and a characterisation of the equality case in terms of our new classes of manifolds introduced in this paper. 

\begin{Rem} Since $h_{BC}^{p,\,q}(X)=h_{A}^{n-p,\,n-q}(X)$ by duality, one gets $h_{BC}(X)=h_A(X)$ and conjugation yields $h_\partial(X)=h_{\bar\partial}(X)$. Therefore, one can replace (\ref{ATin}) with the equivalent inequalities $h_{BC}(X)\geq h_{\delbar}(X)\geq b(X)$ and have the same characterisations for the equality cases.
\end{Rem}

The following general statement is new. Note that it was stated in the introduction as Theorem \ref{The:page-1-ddbar_numerical-char} for any $r-1\in\N$. We now shift $r-1\in\N$ to $r\in\N$ in the notation. In both cases, we implicitly put $\sum_{i= 1}^0e_i(X) = 0$ for the sake of notation consistency.

\begin{The}\label{The:page-r-num} For every compact complex manifold $X$ and for every $r\in\N$, there is an inequality:

         \[
         h_{BC}(X)\geq \sum_{i= 1}^re_i(X)-(r-1)b(X),
         \]

\noindent where $e_i:=\sum_{p,q\in\Z}\dim E_i^{p,q}(X)$.

Moreover, equality holds for some fixed $r\in\N$ if and only if $X$ is a page-$r$-$\del\delbar$-manifold. 

\end{The}

In particular, for $r=1$, we obtain the characterisation of the equality case in $(\ast)$.

Using the upper-semicontinuity of $h_{BC}$ and $h_A$ in families of manifolds, we infer from Theorem \ref{The:page-r-num} applied with $r=1$ the stability of page-$1$-$\del\delbar$-manifolds with fixed Hodge numbers under small deformations of the complex structure. The analogous statement for $r\geq 2$ and constant $e_i$'s with $i\leq r$ also holds.

\begin{Cor}\label{cor: page-1-defs-hodge-cst} If $X_0$ is a {\bf page-$1$-$\del\delbar$-manifold}, then every sufficiently small deformation $X_t$ of $X_0$ which satisfies $h_{\delbar}(X_t)=h_{\delbar}(X_0)$ is again {\bf page-$1$-$\del\delbar$}.

\end{Cor}

If one drops the condition on constant Hodge numbers, one cannot say much in general. In fact, as we will see, the Iwasawa manifold is page-$1$-$\del\delbar$, but any small deformation with different Hodge numbers is not.\\

In order to prove Theorem \ref{The:page-r-num} we will work with abstract (bounded) double complexes rather than double complexes of forms and prove the following (more general) statement. \\

\noindent {\it For a bounded double complex $A$ with finite-dimensional cohomology, let $_1e_i(A)$, resp. $_2e_i(A)$, be the total dimension of the $i$-th page of the row, resp. column, spectral sequence. There is always an inequality: $$h_{BC}(A)+h_{A}(A)\geq \sum _{i=1}^r ({}_1e_i(A)+{}_2e_i(A))-2(r-1)b(A)$$ and the equality is equivalent to the page-$r$-$\del_1\del_2$ property for $A$.} \\

\noindent {\it Proof of Theorem \ref{The:page-r-num}.} Let us write as a shorthand 

\[
LHS:=h_A+h_{BC} \hspace{3ex} \text{and} \hspace{3ex} RHS_r:=\sum _{i=1}^r ({}_1e_i+{}_2e_i))-2(r-1)b.
\]

Before spelling out the details, we state the general idea, which is very simple: In any decomposition of $A$ into indecomposables, $LHS$ counts all the zigzags occuring in $A$, weighted by their length, except for the dots, which are counted twice. When $r=1$, $RHS_1$ counts all the zigzags twice. For an arbitrary $r$, the count on the right becomes slightly more involved.\\

For the actual proof, we first notice that both $LHS$ and $RHS_r$ are additive under direct sums. Therefore, as in the proof of Theorem \ref{The: Characterisation zigzags}, using Theorem \ref{The:Decomposition}, we may reduce the problem to checking the statement on every possible indecomposable double complex individually. Let us run through the list of Theorem \ref{The:Decomposition} using the notation introduced there.

\vspace{1ex}

$(1)$: For a square $S$, we have already noticed in the proof of Theorem \ref{The: Characterisation zigzags} that one has ${}_1e_i(S)={}_1e_i(S)=b(S)=0$, so $RHS_r=0$ for any $r$. On the other hand, on $S$ we have: $\ker\del_1\delbar=\langle \del_1 a,\del_2 a,\del_1\del_2 a\rangle=\im \del_1+\im\del_2$ and $\ker\del_1\cap\ker\del_2=\langle \del_1\del_2 a\rangle=\im\del_1\del_2$, so $h_{BC}(S)=h_A(S)=0$ and $LHS(S)=0=RHS_r(S)$.

\vspace{1ex}

 $(2)$ \& $(3)$: For any even length zigzag $Z$ of length $l=2k$, we saw in the proof of Theorem \ref{The: Characterisation zigzags} that $b(Z)=0$ and

	\[
	_1e_i(Z)+{}_2e_i(Z)=\begin{cases}
    2 &\text{for }i\leq k\\
    0&\text{otherwise.}
    \end{cases}
    \]

Therefore, \[
RHS_r(Z)=\min\{2r,l\}.\]

\vspace{1ex}

 $(4)$ \& $(5)$: For an odd length zigzag, we also saw that ${}_1e_i(Z)={}_2e_i(Z)=b(Z)=1$ for all $i$, so $RHS_r(Z)=2$, for any $r$.\\

It remains to calculate $LHS$ for zigzags. We distinguish two cases:

For a dot $D$, one has $h_{A}(D)=h_{BC}(D)=1$, so $LHS(D)=2=RHS_r(D)$ for any $r$.

For a zigzag $Z$ of length $l\geq 2$, the Aeppli cohomology is the space of generators $H_A(Z)=\langle \{a_i\}_i\rangle$, i.e. all those spaces lying on the first nonzero anti-diagonal of $Z$, and the Bott-Chern cohomology is the space of their images $H_{BC}(Z)=\langle\del_1 a_i, \del_2 a_i\rangle$, i.e. all those spaces lying on the second anti-diagonal. For example,

\[
H_{A}\left(
\begin{tikzcd}
\langle a_1\rangle\ar[r]&\langle \del_1 a_1\rangle\\
&\langle a_2\rangle\ar[u]
\end{tikzcd}\right)=
\begin{tikzcd}
\langle a_1\rangle&0\\
0&\langle a_2\rangle\end{tikzcd}\quad\text{and} \quad
H_{BC}\left(
\begin{tikzcd}
\langle a_1\rangle\ar[r]&\langle \del_1 a_1\rangle\\
&\langle a_2\rangle\ar[u]
\end{tikzcd}\right)=
\begin{tikzcd}
0&\langle\del_1 a_1\rangle\\
0&0.
\end{tikzcd}
\]

\noindent Thus, one has $LHS(Z)=l$.

\vspace{1ex}

Summing up, we see that for any indecomposable complex $I$, one always has $LHS(I)\geq RHS_r(I)$, but equality only holds for squares, dots and even length zigzags of length $\leq 2r$. Using the characterisation of the page-$r$-$\del\delbar$-property given in Theorem \ref{The: Characterisation zigzags}, this completes the proof.\hfill$\Box$\\

\section{Examples of page-$r$-$\partial\bar\partial$-manifolds and counterexamples}\label{subsection: examples}

We shall organise our examples in several classes, each flagged by a specific heading.

\subsection{The case $r=0$ and low dimensions}\label{subsection:general-fact} The first observation is the following rewording of (5.21) in [15].

\begin{Prop}\label{Prop:dd-bar_page_0_equiv} For any compact complex manifold $X$, the following equivalence holds:

  \vspace{1ex}

 \hspace{9ex}  $X$ is a {\bf $\partial\bar\partial$-manifold} $\iff$ $X$ is a {\bf page-$0$-$\partial\bar\partial$-manifold}.

\end{Prop}

In dimensions one and two, it follows from well-known results that the only possible examples of page-$r$-$\del\delbar$-manifolds are K\"ahler:

\begin{Obs} Any compact complex curve is K\"ahler, hence a $\del\delbar$-manifold.  A compact complex surface is a page-$r$-$\del\delbar$-manifold (for some $r$) if any only if it is K\"ahler.

\end{Obs}

\noindent {\it Proof.} It is standard that the Fr\"olicher spectral sequence of any compact complex surface degenerates at $E_1$. It is equally standard that $H_{dR}^k$ is always pure for $k=0,2,4$, while it follows from the Buchdahl-Lamari results (see [10] and [22]) that $H_{dR}^1$ (and hence $H_{dR}^3$) is pure iff the surface is K\"ahler.  \hfill $\Box$

\subsection{Case of the Iwasawa manifold and its small deformations}\label{subsection:Iwasawa-def}

Recall that the Iwasawa manifold $I^{(3)}$ is the nilmanifold of complex dimension $3$ obtained as the quotient of the Heisenberg group of $3\times 3$ upper triangular matrices with entries in $\C$ by the subgroup of those matrices with entries in $\Z[i]$.

It is well known that the Iwasawa manifold is not a $\partial\bar\partial$-manifold. In fact, its Fr\"olicher spectral sequence is known to satisfy $E_1\neq E_2=E_\infty$. On the other hand, it is known that the de Rham cohomology of the Iwasawa manifold can be generated in every degree by de Rham classes of ($d$-closed) pure-type forms. (See e.g. [2].) Together with Cor. \ref{Cor: pure and full} this yields

\begin{Prop}\label{Prop:Iwasawa_page_1_ddbar} The {\bf Iwasawa manifold} is a {\bf page-$1$-$\partial\bar\partial$-manifold}.

\end{Prop}

However, the situation is more complex for the small deformations of the Iwasawa manifold, all of which are already known to not be $\partial\bar\partial$-manifolds. The following result shows, in particular, that unlike the $\partial\bar\partial$-property, the page-$1$-$\partial\bar\partial$-property is {\it not deformation open}.

\begin{Prop}\label{Prop:Iwasawa_page_r_ddbar_small-def} Let $(X_t)_{t\in B}$ be the Kuranishi family of the Iwasawa manifold $X_0$ (see [29], [2]). For every $t\in B$, we have:
\begin{enumerate}
	\item[(1)] $X_t$ is a {\bf page-$1$-$\partial\bar\partial$-manifold} if and only if $X_t$ is {\bf complex parallelisable} (i.e. lies in Nakamura's class (i));

\item[(2)] if $X_t$ lies in one of Nakamura's classes (ii) or (iii), the de Rham cohomology of $X_t$ is {\bf not pure}, so $X_t$ is {\bf not} a {\bf page-$r$-$\partial\bar\partial$-manifold} for any $r\in\N$.
\end{enumerate}

\end{Prop}

\noindent {\it Proof.} That deformations in Nakamura's class $(i)$ are page-$1$-$\partial\bar\partial$-manifolds can be proved in the same way as the Iwasawa manifold was proved to have this property in Proposition \ref{Prop:Iwasawa_page_1_ddbar}. This fact also follows from the far more general Proposition \ref{The:nilmanifolds_c-par} since all the small deformations $X_t$ of $X_0$ are nilmanifolds.

 To show $(2)$, we will actually prove a slightly more general result. Calculations of Angella [2] show that the hypotheses of the next Lemma are satisfied in this case. \hfill $\Box$

\begin{Lem}

Let $X$ be a compact complex manifold with $b_1=4$, $h_{\delbar}^{1,0}=h_{\delbar}^{0,1}=2$ and $h_{A}^{1,0}=3$. Then, either $H^1_{dR}(X,\C)$ or $H^2_{dR}(X,\C)$ is not pure.

\end{Lem}

\noindent {\it Proof.} The proof is combinatorial. We will exploit the fact that the de Rham, Dolbeault and Aeppli cohomologies of indecomposable complexes are computable. This is spelt out in detail in [38]. Summarised briefly, an even-length zigzag has a nonzero differential in the Fr\"olicher spectral sequence or its conjugate, but has no de Rham cohomology. Meanwhile, odd-length zigzags have no differentials in the Fr\"olicher spectral sequence, but have a nonzero de Rham cohomology and $h_A^{p,q}$ counts the zigzags that have a nonzero component in degree $(p,q)$ with possibly outgoing but no incoming arrows.

Specifically, denote by $A=(C^\infty_{p,\,q}(X),\del,\delbar)$ the double complex of $\C$-valued forms on $X$. We investigate for which zigzag shapes $S$ with $(1,0)\in S$ or $(0,1)\in S$, one can have $\mult_S(A)\neq 0$. Assume $H_{dR}^1(X,\C)$ is pure. That means that any odd zigzag contributing to the de Rham cohomology $H^1_{dR}(X,\C)$ is of length one, i.e. drawing only the odd zigzags and not squares or even ones, the lower part of the double complex looks like this:

\[
\parbox{3cm}{\begin{tikzpicture}[scale=0.5]
		
		\draw[->, thick] (0,0) -- (6.5,0) node[anchor=west] {$p$};
		\draw (0,2) -- (4.5,2);
		\draw (0,4) -- (2.5,4);
		
		\draw[->, thick] (0,0) -- (0,6.5) node[anchor=south]{$q$};
		\draw (2,0) -- (2,4.5);
		\draw (4,0) -- (4,2.5);
		
		\draw (1,-.5) node {0};
		\draw (3,-.5) node {1};
		\draw (5,-.5) node {2};
		
		\draw (-.5,1) node {0};
		\draw (-.5,3) node {1};
		\draw (-.5,5) node {2};
		
		\node (11) at (4,4){$\iddots$};
		\node (00) at (1,1) {$\bullet$};
		\node (01) at (1,2.5) {$\bullet$};
		\node (10) at (2.5,1) {$\bullet$};
		\node (10') at (3.5,1) {$\bullet$};
		\node (01') at (1,3.5) {$\bullet$};
		
\end{tikzpicture}}
\]
Here, a $\bullet$ denotes a zigzag of length one and multiplicity one. The symmetry along the diagonal comes from the real structure of $A$ given by complex conjugation. A priori, there may be other zigzags passing through $(1,0)$ and $(0,1)$. Schematically, these would all arise by choosing some connected subgraph with at least one arrow of the diagram

\[
\begin{tikzcd}
	\bullet & &\\
	\bullet^{0,1}\ar[u,dashed,"\delbar"]\ar[r,dashed,"\del"]&\bullet&\\
	&\bullet^{1,0} \ar[u,dashed,"\delbar"]\ar[r,dashed,"\del"]&\bullet
\end{tikzcd}
\]

They could either be of even length or of odd length but not contributing to $H^1_{dR}(X,\C)$ but to $H^2(X,\C)$. Note that the subdiagram 

\[
\begin{tikzcd}
	\bullet^{1,0}\ar[r,"\del"]&\bullet\\
	&\bullet^{0,1}\ar[u,"\delbar"]
\end{tikzcd}
\]
is not allowed since this would be give rise to a nonzero class in $H^1_{dR}(X,\C)$, which we have ruled out already by purity.

However, since $h^{1,0}_{\delbar}+h_{\delbar}^{0,1}=b_1$, there can be no differentials in the Fr\"olicher spectral sequence starting or ending in degree $(1,0)$ or $(0,1)$. In terms of zigzags, this means no even-length zigzag passes through these bidegrees. This rules out the zigzags
\[
\begin{tikzcd}
	\bullet^{0,1}\ar[r,"\del"]&\bullet&\\
	&\bullet^{1,0} \ar[u,"\delbar"]\ar[r,"\del"]&\bullet
\end{tikzcd}
\qquad
\begin{tikzcd}
	\bullet\ar[r,"\del"]&\bullet^{1,0}
\end{tikzcd}
\qquad
\begin{tikzcd}
	\bullet^{1,0}\ar[r,"\del"]&\bullet
\end{tikzcd}
\]
and their reflections along the diagonal (which have to occur with the same multiplicity since $A$ is equipped with a real structure).
So, the only options for zigzags passing through $(1,0)$ that are left are
\[
\begin{tikzcd}
	\bullet & &\\
	\bullet^{0,1}\ar[u,"\delbar"]\ar[r,"\del"]&\bullet&\\
	&\bullet^{1,0} \ar[u,"\delbar"]\ar[r,"\del"]&\bullet
\end{tikzcd}
\quad\text{ or }\qquad 
\begin{tikzcd}
	\bullet&\\
	\bullet^{1,0} \ar[u,"\delbar"]\ar[r,"\del"]&\bullet
\end{tikzcd}
\]
and one of these has to occur since otherwise $H_{A}^{1,0}$ would be of dimension $2$, contradicting the assumptions. But the occurence of either one implies that $H^2_{dR}(X,\C)$ is not pure.
\hfill $\Box$

\subsection{Case of complex parallelisable nilmanifolds}\label{subsection:complex-par}

We will now prove that all complex parallelisable nilmanifolds are page-$1$-$\partial\bar\partial$-manifolds. On the one hand, this generalises one implication in $(1)$ of Proposition \ref{Prop:Iwasawa_page_r_ddbar_small-def}. On the other hand, it provides a large class of page-$1$-$\partial\bar\partial$-manifolds that are not $\partial\bar\partial$-manifolds. Indeed, it is known that a nilmanifold $\nilm$ is never $\partial\bar\partial$ (or even formal in the sense of [15]) unless it is K\"ahler (i.e. a complex torus, or equivalently, the Lie group $G$ is abelian) [19].

Recall that a compact complex parallelisable manifold $X$ is a manifold whose holomorphic tangent bundle is trivial. By Wang's theorem [41], $X$ is a quotient $\nilm$ of a {\it complex} Lie group $G$ by a co-compact, discrete subgroup $\Gamma$. When $G$ is nilpotent, the manifold $X$ is a complex parallelisable nilmanifold. The Iwasawa manifold is an example of this type. We first need an algebraic result.

\begin{Lem}\label{Lem tensor product page 1 ddbar}

	Let $(A^\Cdot,d_A)$ and $(B^\Cdot,d_B)$ be two complexes of vector spaces and $C=A\otimes B$ their tensor product, considered as a double complex, i.e.:

	\begin{align*}
	C^{p,q}&:=A^p\otimes B^q\\
	\del_1 (a\otimes b)&:=d_Aa\otimes b\\
	\del_2(a\otimes b)&:=(-1)^{|a|} a\otimes d_B b
	\end{align*}

Then $C$ satisfies the page-$1$-$\del_1 \del_2$-property.

\end{Lem}

\begin{proof}

	First, we compute the first and second pages of the column Fr\"olicher spectral sequence. (We only treat the column case, the row case being analogous.) The first page is the column cohomology:

	\[
	(_1E_1^{\Cdot,\Cdot},d_1)=(H^q(C^{p,\Cdot},\del_2),\del_1 )
	\]

	Since $\del_2$ is, up to sign, $\Id_A\otimes d_B$, one has $H^q(C^{p,\Cdot},\del_2)=A^p\otimes H^q(B,d_B)$ and $d_1=d_A\otimes\Id_{H(B)}$. Therefore, $_2E_2^{p,q}=H^p(A,d_A)\otimes H^q(B,d_B)$. Now, for every $d_A$-closed element $a\in A^p$ and every $d_B$-closed element $b\in B^q$, the element $a\otimes b\in C^{p+q}$ is $d=\del_1 +\del_2$ closed. Similarly, if one of the two is $d_A$ or $d_B$ exact, the form $a\otimes b$ will be $d$-exact. Hence we get a natural map $\bigoplus_{p+q=k}H^p(A,d_A)\otimes H^q(B,d_B)\rightarrow H_{dR}^k(C)$. Since we are working over a field, the K\"unneth formula tells us that this is an isomorphism.\end{proof}

Given a complex parallelisable nilmanifold $\nilm$, let $\frg$ be the (real) Lie algebra of $G$, and denote by $J\colon \frg \rightarrow \frg$ the endomorphism induced by the complex structure of the Lie group $G$. Then $J^2=-{\rm Id}$ and

\begin{equation}\label{j}
[Jx,y]=J[x,y],
\end{equation} for all $x,y\in \frg$. Let $\gc^*$ be the dual of the complexification $\gc$ of $\frg$ and denote by $\frg^{1,0}$ (respectively $\frg^{0,1}$) the eigenspace of the eigenvalue $i$ (resp. $-i$) of $J$ considered as an endomorphism of $\gc^*$. Condition \eqref{j} is equivalent to $[\frg^{0,1},\frg^{1,0}]=0$ which is equivalent to $d(\frg^{1,0}) \subset \bigwedge\,^{\!\!\!2}\,(\frg^{1,0})$, i.e. there is no component of bidegree $(1,1)$. Therefore, $\bar\partial$ is identically zero on $\bigwedge\,^{\!\!\!p}\,(\frg^{1,0})$ and $\partial$ is identically zero on $\bigwedge\,^{\!\!\!q}\,(\frg^{0,1})$, that is,

\begin{equation}\label{rels}
\partial_{\mid_{\bigwedge\,^{\!\!\!p}(\frg^{1,0})}} = d_{\mid_{\bigwedge\,^{\!\!\!p}(\frg^{1,0})}},
\quad \quad
\bar\partial_{\mid_{\bigwedge\,^{\!\!\!p}(\frg^{1,0})}} = 0,
\quad \quad
\partial_{\mid_{\bigwedge\,^{\!\!\!q}(\frg^{0,1})}} = 0,
\quad \mbox{ and } \quad
\bar\partial_{\mid_{\bigwedge\,^{\!\!\!q}(\frg^{0,1})}} = d_{\mid_{\bigwedge\,^{\!\!\!q}(\frg^{0,1})}}.
\end{equation}

\begin{The}\label{The:nilmanifolds_c-par} Complex parallelisable nilmanifolds are page-$1$-$\partial\bar\partial$-manifolds.

\end{The}

\begin{proof} Sakane [36] showed that the inclusion of the double complex $(\bigwedge^{\Cdot,\Cdot}\gc^*,\del,\delbar)$ as left invariant forms into the complex of all forms on $\Gamma\backslash G$ induces an isomorphism of the respective first pages of the corresponding Fr\"olicher spectral sequences (hence of all later pages). But the equations \eqref{rels} mean that the double complex $(\bigwedge^{\Cdot,\Cdot}\gc^*,\del,\delbar)$ is the tensor product of the simple complexes $(\bigwedge^\Cdot\frg^{1,0},d)$ and $(\bigwedge^\Cdot,\frg^{0,1},d)$, so we can apply Lemma \ref{Lem tensor product page 1 ddbar}.

\end{proof}

\subsection{Nilmanifolds with abelian complex structures}

In this subsection, we construct two classes of page-$1$-$\del\delbar$-manifolds which are not biholomorphic to complex parallelisable nilmanifolds (see Remarks~\ref{Dolb-A1-A2} and~\ref{sGG-A1-A2}). Indeed, they are nilmanifolds endowed with an invariant complex structure that is \emph{abelian}, which means that, in contrast to \eqref{rels}, $\partial$ vanishes on left-invariant $(p,0)$-forms, i.e. $\partial_{\mid_{\bigwedge\,^{\!\!\!p}(\frg^{1,0})}} = 0$.

\begin{The}\label{AbelianNilmanifolds}

	Let $n\geq 3$ and $G$ be the nilpotent Lie group with abelian complex structure defined by the structure equations

\vskip.2cm

$(Ab1_n)$	\  $d\omega^1=0, \ d\omega^2=0,\ 	d\omega^3=\omega^2\wedge\overline{\omega^1},\ldots,\ d\omega^n=\omega^{n-1}\wedge\overline{\omega^1}$,

\vskip.2cm

\noindent or

\vskip.2cm

$(Ab2_n)$ \ $d\omega^1=0,\ldots,\ d\omega^{n-1}=0,\ 	d\omega^{n}=\omega^1\wedge\overline{\omega^2} +\omega^3\wedge\overline{\omega^4}+\cdots+\omega^{n-2}\wedge\overline{\omega^{n-1}}$\ \  (only for odd $n\geq 3$).

\vskip.2cm

\noindent Then, any nilmanifold $\Gamma\backslash G$ is a page-$1$-$\del\delbar$-manifold.

\end{The}

\begin{proof} For every $1\leq k\leq n$, we write $\omega^k=e^k+i\,f^k$, where $e^k$ and $f^k$ are the real part and the imaginary part of the complex (1,\,0)-form $\omega^k$, respectively. We start by working out the real structure equations of the Lie group $G$ in the basis of left-invariant 1-forms $\{e^1,f^1,\ldots,e^{n},f^{n}\}$.
\vspace{1ex}

In the case $(Ab1_n)$, we first notice that, for $3\leq k\leq n$,
$$
\omega^{k-1}\wedge\overline{\omega^1}=(e^{k-1}+i\,f^{k-1})\wedge(e^1-i\,f^1)=-(e^{1}\wedge e^{k-1}+f^{1}\wedge f^{k-1})-i\,(e^{1}\wedge f^{k-1}-f^{1}\wedge e^{k-1}).
$$ 
Hence, the real structure equations are, for $3\leq k\leq n$,
$$
\left\{\begin{array}{ccl}
d e^1 \!\!&\!\!=\!\!&\!\! d f^1=d e^2=d f^2=0,\\[4pt]
d e^k \!\!&\!\!=\!\!&\!\! - e^{1}\wedge e^{k-1} - f^{1}\wedge f^{k-1},\\[4pt]
d f^k \!\!&\!\!=\!\!&\!\! - e^{1}\wedge f^{k-1} + f^{1}\wedge e^{k-1}.
\end{array}\right.
$$
Since the structure constants in this basis are $0,\pm 1$, in particular rational numbers, a result by Mal'cev [26] implies the existence of a lattice $\Gamma$ for $G$. Thus, we get a nilmanifold $\Gamma\backslash G$ endowed with an abelian complex structure.

\vspace{1ex}

We now consider the case $(Ab2_n)$. Writing $n=2m+1$, the last complex equation in $(Ab2_n)$ is $d\omega^{2m+1}=\sum_{k=1}^m \omega^{2k-1}\wedge \overline{\omega^{2k}}$. A direct calculation shows that the real structure equations of the Lie group in the basis of left-invariant 1-forms $\{e^1,f^1,\ldots,e^{n},f^{n}\}$ are

$$
\left\{\begin{array}{ccl}
d e^1 \!\!&\!\!=\!\!&\!\! d f^1=\cdots=d e^{2m}=d f^{2m}=0,\\[4pt]
d e^{2m+1} \!\!&\!\!=\!\!&\!\! \sum_{k=1}^m (e^{2k-1}\wedge e^{2k} + f^{2k-1}\wedge f^{2k}),\\[4pt]
d f^{2m+1} \!\!&\!\!=\!\!&\!\! \sum_{k=1}^m (-e^{2k-1}\wedge f^{2k} + f^{2k-1}\wedge e^{2k}).
\end{array}\right.
$$

Again, the structure constants in this basis are $0,\pm 1\in \Q$, so Mal'cev's theorem [26] implies the existence of a lattice $\Gamma$ for $G$. This induces a nilmanifold $\Gamma\backslash G$ endowed with an abelian complex structure.

\vspace{1ex}

Recall that for abelian complex structures, just as for complex parallelisable ones, Dolbeault, Aeppli and Bott-Chern cohomology groups can be computed using only left-invariant forms (see [12, Remark 4] and [2, Theorem 3.8]).

\vspace{1ex}

$(Ab1_n)$: First, we consider a complex structure defined by $(Ab1_n)$. Let $A$ be the exterior algebra over the vector space $\langle \omega^1, \ldots,\ \omega^n,\overline{\omega^1},\ldots,\ \overline{\omega^n}\rangle$. It identifies naturally with the space of left-invariant $\C$-valued forms on $G$. We write $A_{1}$ for $(A,\partial,\bar\partial)$, where the exterior algebra $A$ is equipped with the differentials defined under $(Ab1_n)$ in the statement. We can also equip $A$ with a different differential $d_{P_{1}}$, acting as follows in degree $1$:

\[
d_{P_{1}}(\omega^1)=0,\ d_{P_{1}}(\omega^2)=0,\ d_{P_{1}}(\omega^3)=\omega^{2}\wedge \omega^{1},\ldots,\  d_{P_{1}}(\omega^n)=\omega^{n-1}\wedge\omega^1.
\]

One has $d_{P_{1}}^2=0$, $d_{P_{1}}=\partial_{P_{1}}+\bar\partial_{P_{1}}$, where $\partial_{P_{1}}$ and $\bar\partial_{P_{1}}$  denote the components of bidegrees $(1,0)$ and $(0,1)$, as well as $d_{P_{1}}(A^{1,0})\subseteq A^{2,0}$. So, $(A,d_{P_{1}})$ can be considered as the space of left-invariant forms on a  nilmanifold endowed with a complex parallelisable structure ${P_{1}}$. By Theorem \ref{The:nilmanifolds_c-par},  $A_{P_{1}}:=(A,\partial_{P_{1}},\bar\partial_{P_{1}})$ has the page-$1$-$\partial_{P_{1}}\bar\partial_{P_{1}}$-property. So, by Theorem \ref{The:page-r-num}, $h_{BC}(A_{P_{1}})+h_A(A_{P_{1}})=h_{\partial_{P_{1}}}(A_{P_{1}})+h_{\bar\partial_{P_{1}}}(A_{P_{1}})$.

Define a $\C$-linear involution $C:A\rightarrow A$ in degree $1$ by $C(\omega^1)=\overline{\omega^1}$ and $C(\omega^i)=\omega^i$, $C(\overline{\omega^i})=\overline{\omega^i}$ for $i>1$ and in degree $k$ by $C(\alpha^1\wedge\ldots\wedge\alpha^k):=C(\alpha^1)\wedge\ldots\wedge C(\alpha^k)$. This is compatible with the total degree, but not with the bigrading. One checks that 

$$
C\circ \partial=\bar\partial_{P_{1}}\circ C\quad\mbox{and}\quad C\circ \bar\partial=\partial_{P_{1}}\circ C.
$$
Indeed, this holds in degree $1$ and then, thanks to the Leibniz rule, in higher degrees as well. Consequently, $C$ induces isomorphisms:
\[
H_{BC}(A_1)\cong H_{BC}(A_{P_{1}}),\qquad H_{A}(A_1)\cong H_{A}(A_{P_{1}}),\qquad H_{\partial}(A_1)\cong H_{\bar\partial_{P_{1}}}(A_{P_{1}}),\qquad H_{\bar\partial}(A_1)\cong H_{\partial_{P_{1}}}(A_{P_{1}}).
\]
at the level of the total cohomologies. For example, the notation means that \[H_{BC}(A_1)=\frac{\ker\partial\cap\ker\bar\partial}{\operatorname{im}\partial\bar\partial}(A_1)=\oplus_{p,q}H_{BC}^{p,q}(A_1).\] We stress that the induced maps are not assumed to be compatible with the bigrading. The existence of such isomorphisms implies that we also have an equality $h_{BC}(A_1)+h_{A}(A_1)=h_{\partial}(A_1)+h_{\bar\partial}(A_1)$, i.e. the page-$1$-$\partial\bar\partial$-property holds for the space of left-invariant forms on $G$. Since $G$ carries an abelian complex structure, this implies that $\Gamma\backslash G$ is a page-$1$-$\partial\bar\partial$-manifold.\\

$(Ab2_n)$: Second, we consider a complex nilmanifold, of odd complex dimension $n\geq 3$, defined by $(Ab2_n)$.
We write $A_{2}$ for $(A,\partial,\bar\partial)$, where the exterior algebra $A$ is now equipped with the differentials defined under $(Ab2_n)$ in the statement.

As before, we may also equip $A$ with a different differential $d_{P_{2}}$, acting as follows in degree $1$:

\[
d_{P_{2}}(\omega^1)=0,\ldots,\ d_{P_{2}}(\omega^{n-1})=0,\ d_{P_{2}}(\omega^n)=\omega^1\wedge{\omega^2} +\omega^3\wedge{\omega^4}+\cdots+
\omega^{n-2}\wedge{\omega^{n-1}}.
\]

One has $d_{P_{2}}(A^{1,0})\subseteq A^{2,0}$, so $(A,d_{P_{2}})$ can be considered as the space of left-invariant forms on a nilmanifold endowed with a complex parallelisable structure ${P_{2}}$. Hence,  $A_{P_{2}}:=(A,\partial_{P_{2}},\bar\partial_{P_{2}})$ has the page-$1$-$\partial_{P_{2}}\bar\partial_{P_{2}}$-property, by Theorem \ref{The:nilmanifolds_c-par}, so we have $h_{BC}(A_{P_{2}})+h_A(A_{P_{2}})=h_{\partial_{P_{2}}}(A_{P_{2}})+h_{\bar\partial_{P_{2}}}(A_{P_{2}})$, by Theorem~\ref{The:page-r-num}.

Let us define a $\C$-linear involution $C:A\rightarrow A$ in degree $1$ by 
$$
C(\omega^{2i+1})=\omega^{2i+1}, \quad 
C(\overline{\omega^{2i+1}})=\overline{\omega^{2i+1}}, \quad  C(\omega^{2i})=\overline{\omega^{2i}},\ \ \mbox{ for } 0\leq i\leq \frac{n-1}{2},
$$
together with $C(\omega^{2n})=\omega^{2n}$ and $C(\overline{\omega^{2n}})=\overline{\omega^{2n}}$. We extend $C$ to  degree $k$ by $C(\alpha^1\wedge\ldots\wedge\alpha^k):=C(\alpha^1)\wedge\ldots\wedge C(\alpha^k)$. One checks that $C\circ \partial=\bar\partial_{P_{2}}\circ C$ and $C\circ \bar\partial=\partial_{P_{2}}\circ C$, so we can conclude as before.

\end{proof}

\begin{Rem}\label{Dolb-A1-A2}

Note that $C\circ d=d_{P_{1}}\circ C$ (similarly for $P_{2}$), and $C$ is compatible with the real structure. So it induces an isomorphism of the underlying real Lie groups. However, the corresponding complex nilmanifolds are not biholomorhic. Indeed, the Hodge number of bidegree $(1,0)$ is given by 

$$
h^{1,0}_{\bar\partial}=2 \ \mbox{ for }\ (Ab1_n),\quad  \mbox{ and }\quad   h^{1,0}_{\bar\partial}=n-1 \ \mbox{ for }\ (Ab2_n),
$$
whereas $h^{1,0}_{\bar\partial_{P}}=n$ for any complex parallelisable nilmanifold of complex dimension $n$.

\end{Rem}

Note that the abelian complex structures defined by $(Ab1_n)$ and $(Ab2_n)$ coincide precisely when $n=3$. We denote this common complex structure on $G$ by $\tilde{J}$ and we write $\tilde{X}=(\nilm,\tilde{J})$ for any nilmanifold endowed with the induced complex structure, still denoted by $\tilde{J}$, where $\Gamma\subset G$ is a lattice.

In the following proposition we prove that, in complex dimension 3, the only complex nilmanifolds which are page-$(r-1)$-$\partial\bar\partial$ for some $r\in\N^\star$ are, apart from a torus, the Iwasawa manifold $I^{(3)}$ and the nilmanifolds $\tilde{X}$.

Notice that these results generalise those in $\S.$\ref{subsection:Iwasawa-def}.

\begin{Prop}\label{Prop:clasif-nilmanifolds} Let $X=(\nilm,J)$ be a complex nilmanifold of complex dimension 3, different from a torus, endowed with an {\it invariant} complex structure $J$.

If there exists $r\in\N^\star$ such that $X$ is a page-$(r-1)$-$\partial\bar\partial$-manifold, then $J$ is equivalent to the complex parallelisable structure of $I^{(3)}$ or to the abelian complex structure $\tilde{J}$ defined by $(Ab1_n)$ in Theorem \ref{AbelianNilmanifolds} for $n=3$. In both cases $r=2$, i.e. both of these manifolds are page-$1$-$\partial\bar\partial$-manifolds.

\end{Prop}

\begin{proof}

We already know by Theorems~\ref{The:nilmanifolds_c-par} and~\ref{AbelianNilmanifolds} that $I^{(3)}$ and $\tilde{X}$ are page-$1$-$\partial\bar\partial$-manifolds.

On the other hand, it is proved in [23] that for any other invariant complex structure $J$ (i.e. not equivalent to $\tilde{J}$ or to the complex parallelisable structure of $I^{(3)}$), the nilmanifold $X=(\nilm,J)$ fails to be pure in degree 4 or 5, that is, the direct sum decomposition of Definition~\ref{Def:purity} is not satisfied for $k=4$ or $k=5$ (or both). So, such complex nilmanifolds $X=(\nilm,J)$ are not page-$(r-1)$-$\partial\bar\partial$-manifolds for any $r\in\N^\star$.\end{proof}

\begin{Rem}\label{sGG-A1-A2}

According to [31] and [34], a compact complex manifold $X$ is called an {\bf sGG manifold} if every Gauduchon metric $\omega$ on $X$ is sG, i.e.  $\partial\omega^{n-1}$ is $\delbar$-exact. 

By the numerical characterisation proved in [34, Theorem 1.6], a compact complex manifold is sGG if and only if $b_1=2h^{0,1}_{\delbar}$. For instance, the Iwasawa manifold is sGG (see [34]), and more generally any complex parallelisable nilmanifold is sGG, due to \eqref{rels}.

For the nilmanifolds endowed with the abelian complex structures defined in Theorem \ref{AbelianNilmanifolds}, we have the following Betti and Hodge numbers:  
$$
b_1=4\not= 2n=2\,h^{0,1}_{\bar\partial} \ \mbox{ for }\ (Ab1_n),\quad  \mbox{ and }\quad   b_1=2(n-1)\not= 2n=2\,h^{0,1}_{\bar\partial} \ \mbox{ for }\ (Ab2_n).
$$
Hence, such complex nilmanifolds are not sGG-manifolds.

On the other hand, all the sGG nilmanifolds of complex dimension $n=3$ are identified in [34, Theorem 6.1]. In particular, there exist complex nilmanifolds different from the Iwasawa manifold and $\tilde{X}$ which are sGG, so by Proposition~\ref{Prop:clasif-nilmanifolds}, they are not page-$(r-1)$-$\partial\bar\partial$-manifolds for any $r\in\N^\star$.

Therefore, the page-$1$-$\partial\bar\partial$ and the sGG properties of compact complex manifolds are unrelated.

\end{Rem}

\subsection{Nakamura solvmanifolds}

Consider $G:=\C\ltimes_\phi\C^2$, where $\phi$ is either

\[
\phi(z)=\begin{pmatrix}
e^z&0\\
0&e^{-z}
\end{pmatrix}\qquad \text{ or }\qquad \phi(z)=\begin{pmatrix}
e^{\mbox{\tiny Re}(z)}&0\\
0&e^{-\mbox{\tiny Re}(z)}
\end{pmatrix}
\] 

\noindent ({\it complex parallelizable}, resp. {\it completely solvable} case). Define $X$ to be the quotient of $G$ by a lattice of the form $\Gamma\ltimes_\phi\Gamma'$ with $\Gamma\subset \C$, $\Gamma'\subset\C^2$ lattices. These manifolds were studied in [29] and are called {\it Nakamura manifolds}. They are among the best known {\it solvmanifolds}, but are {\it not nilmanifolds}. In [4], Angella and Kasuya computed the Hodge, Bott-Chern and Aeppli numbers depending on the lattice $\Gamma$. (These numbers turn out to be independent of $\Gamma'$). In particular, their calculations yield the equality $h_{BC}(X)=h_{\delbar}(X)$. Hence, by Theorem \ref{The:page-1-ddbar_numerical-char}, we obtain:

\begin{Cor}\label{Cor:A-K_Nakamura}

	The complex parallelisable and completely solvable {\bf Nakamura manifolds} considered in [4] are {\bf page-$1$-$\del\delbar$-manifolds}.

\end{Cor}

\section{Construction methods for page-r-$\del\delbar$-manifolds}\label{section:geom-prop_page-r-ddbar}

Among the issues that we take up in this section, there is the behaviour of page-$r$-$\partial\bar\partial$-manifolds under modifications and its link with the open problem of submanifold heredity of this class of manifolds.

\begin{The}\label{The:geom-prop_page-r-ddbar}Let $X$ and $Y$ be compact complex manifolds.

\begin{enumerate}

\item[(1)] If $X$ is a page-$r$-$\del\delbar$-manifold and $Y$ is a page-$r'$-$\del\delbar$-manifold, the product $X\times Y$ is a page-$\tilde{r}$-$\del\delbar$-manifold, where $\tilde{r}=\max\{r,r'\}$. 

Conversely, if the product is a page-$r$-$\del\delbar$-manifold, so are both factors.

\item[(2)] For any vector bundle $\mathcal{V}$ over $X$, the projectivised bundle $\mathbb{P}(\mathcal{V})$ is a page-$r$-$\del\delbar$-manifold if and only if $X$ is.

\item[(3)] Suppose $X$ is a page-$r$-$\del\delbar$-manifold. Let $f:X\longrightarrow Y$ be a surjective holomorphic map and assume there exists a $d$-closed $(l,l)$-current $\Omega$ on $X$ (with $l=\dim X-\dim Y)$ such that $f_*\Omega\neq 0$. Then $Y$ is also a page-$r$-$\del\delbar$-manifold. 
In particular, this implication always holds when $\dim X=\dim Y$, e.g. for contractions (take $\Omega$ to be a constant).
\item[(4)] Given a submanifold $Z\subset X$, denote by $\widetilde{X}$ the blow-up of $X$ along $Z$. If $X$ is page-$r$-$\del\delbar$ and $Z$ is page-$r'$-$\del\delbar$, then $\widetilde{X}$ is a page-$\widetilde{r}$-$\del\delbar$-manifold, where $\tilde{r}=\max\{r,r'\}$. 
Conversely, if $\widetilde{X}$ is page $r$-$\del\delbar$, so are $X$ and $Z$.
\item[(5)] The page-$r$-$\del\delbar$-property of compact complex manifolds is a bimeromorphic invariant if and only if it is stable under passage to submanifolds.

\end{enumerate}

\end{The}

\noindent {\it Proof.} The proofs are very similar to those in [38, Cor. 28]. We will be using the characterisation of the page-$r$-$\del\delbar$-property in terms of occuring zigzags (Theorem \ref{The: Characterisation zigzags}) and $E_1$-isomorphisms (Def. \ref{Def: E_r-iso}), in particular Lemma \ref{Lem: E_1 isos and zigzags}.

Write $A_X$ as shorthand for the double complex $(C^\infty_{\Cdot,\Cdot}(X,\C),\del,\delbar)$ and $A_X[i]$ for the shifted double complex with bigrading $(A_X[i])^{p,q}:=A_X^{p-i,q-i}$. By [38, Sect.~4], [37] and [27], we have the following {\bf $E_1$-isomorphisms}:\footnote{Cf. also [35], [43], [28] and [6] for different approaches to the blow-up question in the setting of particular cohomologies.}

\begin{align}
A_{X\times Y}&\simeq_1 A_X\otimes A_Y,\\
A_{\Proj (\Vc)}&\simeq_1 \bigoplus_{i=0}^{\rk \Vc-1} A_{X}[i],\\
A_X&\simeq_1 A_Y\oplus A_X/f^*A_Y,\\
A_{\widetilde{X}}&\simeq_1 A_X\oplus \bigoplus_{i=1}^{\operatorname{codim}Z-1} A_Z[i].\label{eqn: blowup}
\end{align}

Since the occuring zigzags get only shifted, $A_X[i]$ satisfies the page-$r$-$\del\delbar$-property if and only if $A_X$ does. Furthermore, a direct sum of complexes satisfies the page-$r$-$\del\delbar$-property if and only if each summand does. So, the second, third and fourth $E_1$-isomorphisms imply $(2)$, $(3)$ and $(4)$

For the first part of $(1)$, we use the first isomorphism and the fact that one knows how irreducible subcomplexes behave under tensor product (see [38, Prop. 16]). In particular, even-length zigzags do not get longer and the product of two length-one zigzags is again of length one. For the converse, note that $A_X$ and $A_Y$ are direct summands in their tensor product, so we can argue as before.

The `if' statement in the last part of $(5)$ is a direct consequence of $(4)$ and the weak factorisation theorem [1], which says that every bimeromorphic map can be factored as a sequence of blow-ups and blow-downs with smooth centres. The `only if' part also follows from $(4)$ (cf. also [27]). Indeed, let $X$ be page-$r$-$\del\delbar$ and let $Z\subset X$ be a submanifold. If $Z$ has codimension one, we replace $X$ by $X\times \Proj^1_\C$ (which is still page-$r$-$\del\delbar$ by $(1)$) and $Z$ by $Z\times\{0\}$. By assumption, the blow-up is still page-$r$-$\del\delbar$ and one can apply $(4)$ to infer that the same holds for $Z$.\hfill$\Box$\\

Since for surfaces and threefolds, the centre of a nontrivial blow-up is a point or a curve, we get

\begin{Cor} Fix any $r\in\N$. The page-$r$-$\del\delbar$-property of compact complex surfaces and threefolds is a bimeromorphic invariant.

\end{Cor}

In the remainder of the paper, because we can handle it with very similar methods, we point out a result about a class of manifolds that contains the class of page-$(r-1)$-$\partial\bar\partial$-manifolds. Given a compact complex $n$-dimensional manifold $X$, recall the following facts and definitions:

\begin{enumerate}
	\item[(1)] For a fixed integer $r\geq 2$ and a bidegree $(p,\,q)$, a $C^\infty$ form $\alpha$ of bidegree $(p,\,q)$ on $X$ is {\bf $E_r$-closed} (in the sense that it represents a cohomology class $\{\alpha\}_{E_r}\in E_r^{p,\,q}(X)$ on the $r$-th page of the Fr\"olicher spectral sequence of $X$) if and only if there exist forms $u_l\in C^\infty_{p+l,\,q-l}(X,\,\C)$ with $l\in\{1,\dots , r-1\}$ such that \begin{equation}\label{eqn:E_r-closedness_end}\bar\partial\alpha=0, \hspace{2ex} \partial\alpha = \bar\partial u_1, \hspace{2ex} \partial u_1 = \bar\partial u_2,\dots , \hspace{2ex}\partial u_{r-2} = \bar\partial u_{r-1}.\end{equation} (See [13], taken up again in [32, Proposition 2.7].)
\item[(2)] For a fixed integer $r\geq 1$, an {\bf $E_r$-sG metric} on $X$ is a Hermitian metric (i.e. a $C^\infty$ positive definite $(1,\,1)$-form) $\omega$ such that $\partial\omega^{n-1}$ is $E_r$-exact. (See [32, Definition 3.1, (i)].)

\item[(3)] For a fixed integer $r\geq 1$, $X$ is said to be an {\bf $E_r$-sG manifold} if an $E_r$-sG metric exists on $X$. (See [32, Definition 3.1, (ii)].)

\item[(4)] For a fixed integer $r\geq 1$, $X$ is said to be an {\bf $E_r$-sGG manifold} if every Gauduchon metric on $X$ is an $E_r$-sG metric. (See [32, Definition 3.1, (iii)].)
\item[(5)] For a fixed integer $r\geq 1$, every {\bf page-$(r-1)$-$\del\delbar$-manifold} is an {\bf $E_r$-sGG manifold}. (See [33, Proposition 5.2].)
\end{enumerate}

Let us now fix an integer $r\geq 1$ and consider the canonical linear map on a compact complex $n$-dimensional manifold $X$: \begin{equation*}T_r:H_A^{n-1,\, n-1}(X,\,\C)\rightarrow E_r^{n,\, n-1}(X), \hspace{5ex} [\alpha]_A\mapsto\{\del\alpha\}_{E_r}.\end{equation*} Note that this map is well defined since:

\vspace{1ex}

Whenever $\alpha$ represents an Aeppli class, we have $\alpha\in\ker(\partial\bar\partial)$, so $\bar\partial(\partial\alpha)=0$ and $\partial(\partial\alpha)=0$. Hence, $\partial\alpha$ satisfies the $E_r$-closedness conditions (\ref{eqn:E_r-closedness_end}) with $u_1=\dots = u_{r-1}=0$. Thus, the class $\{\del\alpha\}_{E_r}$ is well defined.

\vspace{1ex}

If two $(n-1,\,n-1)$-forms $\alpha,\beta\in\ker(\partial\bar\partial)$ represent the same Aeppli class, there exist forms $u\in C^\infty_{n-2,\,n-1}(X,\,\C)$ and $v\in C^\infty_{n-1,\,n-2}(X,\,\C)$ such that $\alpha-\beta=\partial u + \bar\partial v$. Hence, $\partial\alpha-\partial\beta=\partial\bar\partial v$. In particular, $\partial\alpha-\partial\beta\in\mbox{im}\,\bar\partial$. In the language of [32, Definition 3.1], this means that $\partial\alpha-\partial\beta$ is $E_1$-exact, which implies that it is $E_r$-exact (i.e. it represents the {\it zero} class on the $r$-th page of the Fr\"olicher spectral sequence of $X$) for every $r\geq 1$. Thus, $\{\partial\alpha\}_{E_r} = \{\partial\beta\}_{E_r}$. This proves that the map $T_r$ is independent of the choice of representative of the class $[\alpha]_A\in H^{n-1,\,n-1}(X,\,\C)$.

\begin{Lem} An $n$-dimensional compact complex manifold $X$ is $E_r$-sGG if and only if $T_r=0$.

\end{Lem}

\noindent {\it Proof.} The argument is the analogue in this context of the proof of Observation 5.3. in [31].

Let $\omega$ be a Gauduchon metric on $X$. This means that $\omega$ is a Hermitian metric such that $\partial\bar\partial\omega^{n-1}=0$ (cf. [18]). We see that $\omega$ is $E_r$-sG if and only if $[\omega^{n-1}]_A\in \ker T_r$. Thus, the set of all classes $[\omega^{n-1}]_A\in H_A^{n-1,\,n-1}(X,\,\C)$ that are representable by the $(n-1)$-st power $\omega^{n-1}$ of an $E_r$-sG metric $\omega$ is precisely the intersection $${\cal G}_X\cap\ker T_r,$$ where ${\cal G}_X\subset H^{n-1,\,n-1}_A(X,\,\R)$ is the Gauduchon cone of $X$ (defined in [31, Definition 5.1] as the set of all classes $[\omega^{n-1}]_A\in H^{n-1,\,n-1}_A(X,\,\R)$ that are representable by the $(n-1)$-st power $\omega^{n-1}$ of a Gauduchon metric $\omega$).

Since the Gauduchon cone is open in $H^{n-1,\,n-1}_A(X,\,\R)$ and non-empty, the equality $${\cal G}_X\cap\ker T_r = {\cal G}_X$$ (which holds if and only if $X$ is an $E_r$-sGG manifold) is equivalent to $\ker T_r = H_A^{n-1,\,n-1}(X,\,\C)$, so to the map $T_r$ vanishing identically. \hfill $\Box$

\vspace{2ex}

As a consequence of this, we get the {\bf bimeromorphic invariance} of the {\bf $E_r$-sGG} property.

\begin{Cor} Let $X$ and $\widetilde{X}$ be {\bf bimeromorphically equivalent} compact complex manifolds. Then, every Gauduchon metric on $X$ is $E_r$-sG if and only if this is true on $\widetilde{X}$.

\end{Cor}

\noindent {\it Proof.}  By the weak factorisation theorem [1], it suffices again to check this for blow-ups $\widetilde{X}\rightarrow X$ with $d$-dimensional smooth centers $Z\subset X$ of codimension $\geq 2$. After picking any isomorphism realising formula (\ref{eqn: blowup}), any class $c\in H^{n-1,\,n-1}_A(\widetilde{X})$ can be written as $c=c_X+c_Z$, with $c_X\in H^{n-1,\,n-1}_A(X)$ and $c_Z\in H^{d,\,d}_A(Z)$. Hence, $T_r c=T_r c_X + T_rc_Z=T_rc_X$ since $\del\eta=0$ for all $(d,d)$-forms on $Z$ for dimension reasons.\hfill$\Box$\\

Note that the above map $T_r$ is given in all cases by applying $\del$. Generally speaking, if $A=B\oplus C$, then $H_A(A)=H_A(B)\oplus H_A(C)$ and $E_r(A)=E_r(B)\oplus E_r(C)$ and $T_r^A=T_r^B+T_r^C$. We omitted the superscripts on $T_r$ in the above proof for the sake of simplicity.\\

We end the paper with an obvious open problem suggested by the examples constructed so far.

\begin{Prob}
For every $r\geq 2$, construct page-$r$-$\partial\bar\partial$-manifolds that are not page-$(r\!-\!1)$-$\partial\bar\partial$.
\end{Prob}

We believe such examples exist and the difficulty of constructing them for $r\geq 2$ is related to the general difficulty of constructing manifolds with very non-degenerate Fr\"olicher spectral sequence.

\vspace{2ex}

\noindent {\bf Acknowledgements.} This work was partially supported by 
grant PID2020-115652GB-I00, funded by
MCIN/AEI/10.13039/501100011033, and grant E22-17R ``\'Algebra y Geometr\'{\i}a'' (Gob. Arag\'on/FEDER). 
We are grateful to the referee for useful comments that helped us to improve the presentation of the paper.

\vspace{2ex}

\noindent {\bf References.} 

\vspace{2ex}

\noindent [1] 
D. Abramovich, K. Karu, K. Matsuki and J. W\l{}odarczyk, {\it Torification and factorization of birational maps}, J. Amer. Math. Soc. {\bf 15} (2002), no. 3, 531--572.

\vspace{1ex}

\noindent [2] 
D. Angella, {\it The cohomologies of the Iwasawa manifold and of its small deformations}, J. Geom. Anal. {\bf 23} (2013), no. 3, 1355--1378.

\vspace{1ex}

\noindent [3] 
D. Angella and H. Kasuya, {\it Cohomologies of deformations of solvmanifolds and closedness of some properties}, North-West. Eur. J. Math. {\bf 3} (2017), 75--105.

\vspace{1ex}

\noindent [4] 
D. Angella and H. Kasuya, {\it Bott-Chern cohomology of solvmanifolds}, Ann. Global  Anal. Geom. {\bf 52} (2017), no. 4, 363--411.

\vspace{1ex}

\noindent [5] 
D. Angella, A. Otal, L. Ugarte and R. Villacampa, {\it Complex structures of splitting type}, Rev. Mat. Iberoam. {\bf 33} (2017), no. 4, 1309--1350.

\vspace{1ex}

\noindent [6] 
D. Angella, T. Suwa, N. Tardini and A. Tomassini, {\it Note on Dolbeault cohomology and Hodge structures up to bimeromorphisms}, Complex Manifolds {\bf 7} (2020), no. 1, 194--214.

\vspace{1ex}

\noindent [7] 
D. Angella and A. Tomassini, {\it On cohomological decomposition of almost-complex manifolds and deformations}, J. Symplectic Geom. {\bf 9}, no. 3 (2011), 403--428.

\vspace{1ex}

\noindent [8] 
D. Angella and A. Tomassini, {\it On the $\partial\bar\partial$-Lemma and Bott-Chern cohomology}, Invent. Math. {\bf 192} (2013), no. 1, 71--81.

\vspace{1ex}

\noindent [9] 
F. A. Bogomolov, {\it Hamiltonian K\"ahler manifolds}, Soviet Math. Dokl. {\bf 19} (1978), 1462--1465.

\vspace{1ex}

\noindent [10] 
N. Buchdahl, {\it On Compact K\"ahler Surfaces}, Ann. Inst. Fourier (Grenoble) {\bf 49} (1999), no. 1, 287--302.

\vspace{1ex}

\noindent [11] 
F. Campana, {\it The class $\mathcal{C}$ is not stable by small deformations}, Math. Ann. {\bf 290} (1991), no. 1, 19--30.

\vspace{1ex}

\noindent [12] 
S. Console and A. Fino, {\it Dolbeault cohomology of compact nilmanifolds}, Transform. Groups {\bf 6} (2001), no. 2, 111--124.

\vspace{1ex}

\noindent [13] 
L. A. Cordero, M. Fern\'andez, A. Gray and L. Ugarte, {\it A general description of the terms in the Fr\"olicher spectral sequence}, Differential Geom. Appl. {\bf 7} (1997), no. 1, 75--84.

\vspace{1ex}

\noindent [14] 
P. Deligne, {\it Th\'eorie de Hodge: II}, Publ. Math. Inst. Hautes \'Etudes Sci. {\bf 40}, (1971), 5--57.

\vspace{1ex}

\noindent [15] 
P. Deligne, Ph. Griffiths, J. Morgan and D. Sullivan, {\it Real Homotopy Theory of K\"ahler Manifolds}, Invent. Math. {\bf 29}, (1975), 245--274.

\vspace{1ex}

\noindent [16] 
R. Friedman, {\it The $\partial\bar\partial$-lemma for general Clemens manifolds}, Pure. Appl. Math. Q. {\bf 15} (2019), no. 4, 1001--1028.

\vspace{1ex}

\noindent [17] 
A. Fr\"olicher, {\it Relations between the cohomology groups of Dolbeault and topological invariant}, Proc. Nat. Acad. Sci. U.S.A. {\bf 41} (1955), 641--644.

\vspace{1ex}

\noindent [18] 
P. Gauduchon, {\it Le th\'eor\`eme de l'excentricit\'e nulle}, C. R. Acad. Sci. Paris, S\'er. A, {\bf 285} (1977), 387--390.

\vspace{1ex}

\noindent [19] 
K. Hasegawa, {\it Minimal models of nilmanifolds}, Proc. Amer. Math. Soc. {\bf 106} (1989), no. 1, 65--71.

\vspace{1ex}

\noindent [20] 
H. Kasuya and J. Stelzig, {\it Fr\"olicher spectral sequence and Hodge structures on the cohomology of complex parallelisable manifolds}, 
to appear in Transform. Groups.

\vspace{1ex}

\noindent [21] 
M. Khovanov and Y. Qi, {\it A faithful braid group action on the stable category of tricomplexes}, 
SIGMA Symmetry Integrability Geom. Methods Appl. {\bf 16} (2020), Paper No. 019, 32 pp.

\vspace{1ex}

\noindent [22] 
A. Lamari, {\it Courants k\"ahl\'eriens et surfaces compactes}, Ann. Inst. Fourier (Grenoble) {\bf 49} (1999), no. 1, 263--285.

\vspace{1ex}

\noindent [23] 
A. Latorre and L. Ugarte, {\it Cohomological decomposition of complex nilmanifolds}, Topol. Methods Nonlinear Anal. {\bf 45} (2015), no. 1, 215--231.

\vspace{1ex}

\noindent [24] 
C. LeBrun and Y. S. Poon, {\it Twistors, K\"ahler manifolds, and bimeromorphic geometry II}, J. Amer. Math. Soc. {\bf 5} (1992), no. 2, 317--325.

\vspace{1ex}

\noindent [25] 
T.-J. Li and W. Zhang, {\it Comparing tamed and compatible symplectic cones and cohomological properties of almost complex manifolds}, Comm. Anal. Geom. {\bf 17}  (2009), no. 4, 651--683.

\vspace{1ex}

\noindent [26] 
A. I. Mal'cev, {\it On a class of homogeneous spaces},  Amer. Math. Soc. Transl. {\bf 39} (1951), 33 pp.

\vspace{1ex}

\noindent [27] 
L. Meng, {\it The heredity and bimeromorphic invariance of the $\del\delbar$-lemma property}, C. R. Math. Acad. Sci. Paris {\bf 359} (2021), 645???650.

\vspace{1ex}

\noindent [28] 
L. Meng, {\it Leray-Hirsh theorem and blow-up formula for Dolbeault cohomology}, Ann. Mat. Pura Appl. {\bf 199}, (2020), 1997--2014.

\vspace{1ex}

\noindent [29] 
I. Nakamura, {\it Complex parallelisable manifolds and their small deformations}, J. Differential Geometry {\bf 10} (1975), 85--112.

\vspace{1ex}

\noindent [30] 
D. Popovici, {\it Deformation Openness and Closedness of Various Classes of Compact Complex Manifolds; Examples}, Ann. Sc. Norm. Super. Pisa Cl. Sci. (5), Vol. XIII (2014), 255--305.

\vspace{1ex}

\noindent [31] 
D. Popovici, {\it Aeppli Cohomology Classes Associated with Gauduchon Metrics on Compact Complex Manifolds}, Bull. Soc. Math. France {\bf 143} (2015), no. 3, 1--37.

\vspace{1ex}

\noindent [32] 
D. Popovici, {\it Adiabatic Limit and Deformations of Complex Structures}, arXiv e-print AG 1901.04087v2.

\vspace{1ex}

\noindent [33] 
D. Popovici, J. Stelzig and L. Ugarte, {\it Higher-page Bott-Chern and Aeppli Cohomologies and Applications}, J. Reine Angew. Math. {\bf 777} (2021), 157--194.

\vspace{1ex}

\noindent [34] 
D. Popovici and L. Ugarte, {\it Compact complex manifolds with small Gauduchon cone}, Proc. Lond. Math. Soc. {\bf 116} (2018), no. 5, 1161--1186.

\vspace{1ex}

\noindent [35] 
S. Rao, S. Yang and X. Yang, {\it Dolbeault cohomologies of blowing up complex manifolds}, J. Math. Pures Appl. {\bf 130}, (2019), 68--92.

\vspace{1ex}

\noindent [36] 
Y. Sakane, {\it On compact complex parallelisable solvmanifolds}, Osaka Math. J. {\bf 13} (1976), no. 1, 187--212.

\vspace{1ex}

\noindent [37] 
J. Stelzig, {\it The Double Complex of a Blow-up}, 
Int. Math. Res. Not. IMRN 2021, no. 14, 10731--10744.

\vspace{1ex}

\noindent [38] 
J. Stelzig, {\it On the Structure of Double Complexes}, 
J. Lond. Math. Soc. {\bf 104} (2021), no. 2, 956--988.

\vspace{1ex}

\noindent [39] 
G. Tian, {\it Smoothness of the Universal Deformation Space of Compact Calabi-Yau Manifolds and Its Petersson-Weil Metric}, Mathematical Aspects of String Theory (San Diego, 1986), Adv. Ser. Math. Phys. 1, World Sci. Publishing, Singapore (1987), 629--646.

\vspace{1ex}

\noindent [40] 
A. N. Todorov, {\it The Weil-Petersson Geometry of the Moduli Space of $SU(n\geq 3)$ (Calabi-Yau) Manifolds I}, Comm. Math. Phys. {\bf 126} (1989), 325--346.

\vspace{1ex}

\noindent [41] 
H.-C. Wang, {\it Complex Parallelisable Manifolds}, Proc. Amer. Math. Soc. {\bf 5} (1954), 771--776.

\vspace{1ex}

\noindent [42] 
C. C. Wu, {\it On the Geometry of Superstrings with Torsion}, Thesis (Ph.D.), Department of Mathematics, Harvard University, Cambridge MA 02138, (2006).

\vspace{1ex}

\noindent [43] 
S. Yang and X. Yang, {\it Bott-Chern blow-up formula and bimeromorphic invariance of the $\del\delbar$-Lemma for threefolds}, Trans. Amer. Math. Soc. {\bf 373} (2020), 8885--8909.

\vspace{6.5ex}

\noindent Institut de Math\'ematiques de Toulouse,       \hfill Mathematisches Institut

\noindent Universit\'e Paul Sabatier,                    \hfill  Ludwig-Maximilians-Universit\"at

\noindent 118 route de Narbonne, 31062 Toulouse, France  \hfill Theresienstr. 39, 80333 M\"unchen, Germany

\noindent Email: popovici@math.univ-toulouse.fr          \hfill Email: Jonas.Stelzig@math.lmu.de

\vspace{1.5ex}

\noindent and

\vspace{1.5ex}

\noindent Departamento de Matem\'aticas\,-\,I.U.M.A.,

\noindent Universidad de Zaragoza,

\noindent Campus Plaza San Francisco, 50009 Zaragoza, Spain

\noindent Email: ugarte@unizar.es

\end{document}